\documentclass[a4paper]{amsart}
\usepackage{amsmath,amsthm,amssymb,latexsym,epic,bbm,comment,mathrsfs}
\usepackage{graphicx,enumerate,stmaryrd, color}
\usepackage[all,2cell]{xy}
\xyoption{2cell}

\newtheorem{theorem}{Theorem}
\newtheorem{lemma}[theorem]{Lemma}
\newtheorem{corollary}[theorem]{Corollary}
\newtheorem{proposition}[theorem]{Proposition}

\newtheorem{example}[theorem]{Example}

\newtheorem{remark}[theorem]{Remark}
\usepackage[all]{xy}
\usepackage[active]{srcltx}
\usepackage[parfill]{parskip}
\usepackage{enumerate}

\newcommand{\tto}{\twoheadrightarrow}

\begin{document}
\title[Infinite rank module categories]
{Infinite rank module categories over
finite dimensional\\ $\mathfrak{sl}_2$-modules
in Lie-algebraic context}

\author[V.~Mazorchuk and X.~Zhu]
{Volodymyr Mazorchuk and Xiaoyu Zhu}

\begin{abstract}
We study locally finitary realizations of simple transitive module 
categories of infinite rank over the monoidal 
category $\mathscr{C}$ of finite dimensional modules for the complex
Lie algebra $\mathfrak{sl}_2$. Combinatorics of such 
realizations is governed by six infinite Coxeter diagrams.
We show that five of these are realizable in our setup,
while one (type $B_\infty$) is not. We also describe the
$\mathscr{C}$-module subcategories 
of $\mathfrak{sl}_2$-mod generated by simple modules
as well as the $\mathscr{C}$-module categories
coming from the natural action of $\mathscr{C}$
on the categories of finite dimensional modules
over Lie subalgebras of $\mathfrak{sl}_2$.
\end{abstract}

\maketitle

\section{Introduction and description of the results}\label{s1}

\subsection{Motivation and formulation of the problem}\label{s1.1}

Given a finite subgroup $G$ of $SU(2)$, the monoidal category of
finite dimensional $SU(2)$-modules acts, by restriction and 
tensor product, on the category of finite dimensional $G$-modules.
This observation is the basis of the so-called McKay correspondence,
see \cite{Mc}, which, in particular, classifies finite subgroups 
of $SU(2)$ in terms of extended Dynkin diagrams and also connects 
them to orbifold  singularities and simple Lie groups.

Consider the simple complex Lie algebra $\mathfrak{sl}_2$. For this
Lie algebra, we have the monoidal category $\mathscr{C}$ of all 
finite dimensional modules, with the obvious monoidal  structure.
The category $\mathscr{C}$ is the obvious Lie-algebraic analogue
of the category of finite dimensional $SU(2)$-modules.  
The category $\mathscr{C}$ is an interesting and important 
object of study, see \cite{EGNO,EK,EO,KO} and reference therein.

From the point of view of representation theory,
of particular interest is the problem to study
various types of $\mathscr{C}$-module  categories.
The most basic of these are the so-called
{\em simple} $\mathscr{C}$-module  categories
(a.k.a. {\em simple transitive} in the terminology of
\cite{MM5}).
In particular, simple 
 $\mathscr{C}$-module  categories
of finite rank (that is, with finitely many simple 
objects) were classified in \cite{EO} in terms of
certain graphs (and bilinear forms). As was shown in 
\cite{EK}, the combinatorics of a natural subclass of
these categories is governed by extended Dynkin diagrams, 
similarly to the McKay correspondence.

In this paper we are trying to study 
$\mathscr{C}$-module  categories in the context
of Lie algebras and their representations.
One of our starting questions was: What 
kind of  $\mathscr{C}$-module  categories one
obtains if one looks at the action of $\mathscr{C}$
on the category of finite dimensional modules
over some Lie subalgebra  of $\mathfrak{sl}_2$?
This is not very difficult  to answer by direct
computations as all Lie subalgebras of 
$\mathfrak{sl}_2$ are easy to write down.

One immediate observation, as soon as one looks into
examples, is that all  the categories one obtains
in the natural way have infinite rank
(i.e. have infinitely many isomorphism classes
of indecomposable  objects).  At the level of the
Grothendieck ring,  it is known from \cite{EK}
that the combinatorics of such kind of 
categories is governed by the so-called 
infinite Dynkin diagrams,  see  \cite{HPR,HPR2}
and Subsection~\ref{s3.4}.
These are just six infinite diagrams, called
$A_\infty$, $A_\infty^\infty$, $B_\infty$,
$C_\infty$, $D_\infty$ and $T_\infty$.
It is then natural to ask: Do all of them appear?
If one of them  appears, then in how many different 
versions? If one of them  appears, then what
can one say about the properties of the corresponding
$\mathscr{C}$-module  category?
These are the questions which we ask
and address in the present paper.

\subsection{Setup  and  results}\label{s1.2}

For the sake of increased flexibility, we adopt the following
setup: we work with a finite-dimensional Lie algebra 
$\mathfrak{g}$ with a fixed surjective Lie algebra homomorphism
$\varphi:\mathfrak{g}\to\mathfrak{sl}_2$. In this situation,
the pullback along $\varphi$ allows us to consider 
the monoidal category $\mathscr{C}$ as  a monoidal subcategory 
of the category of all  $\mathfrak{g}$-modules.

Further, we are interested in $\mathscr{C}$-module categories
$\mathcal{M}$ that are subcategories of 
$\mathfrak{g}$-Mod, closed under the natural action of 
$\mathscr{C}$, which have the following properties:
\begin{itemize}
\item $\mathcal{M}$ is locally finitary in the sense that
all hom-spaces are finite dimensional;
\item $\mathcal{M}$ is additive, idempotent split
and Krull-Schmidt;
\item $\mathcal{M}$ is cyclic, in the sense that there exists
an object $M\in \mathcal{M}$ such that $\mathcal{M}$
coincides with the additive closure of $\mathscr{C}\cdot M$.
\end{itemize}
Note that, automatically, such $\mathcal{M}$ contains 
at most countably many indecomposable objects, up to isomorphism.
A typical example of such $\mathcal{M}$ is the left regular
$\mathscr{C}$-module category ${}_\mathscr{C}\mathscr{C}$.

In this setup, we obtain the following results about the
realization of combinatorially different types of
$\mathscr{C}$-module categories:
\begin{itemize}
\item Type $A_\infty$ is realized via the left regular
$\mathscr{C}$-module category ${}_\mathscr{C}\mathscr{C}$,
moreover, such realization is unique up to equivalence,
see Proposition~\ref{prop-s4.2-1} and Proposition~\ref{prop-s4.2-2}.
\item Type $A_\infty^\infty$ has uncountably many
pairwise non-equivalent realizations via semi-simple categories
which can be constructed already inside $\mathfrak{sl}_2$-mod, see 
Proposition~\ref{prop-s5.7-1} and Proposition~\ref{prop-s5.8-1}.
Moreover, we also construct a simple transitive realization
of this type via a non semi-simple category, see 
Proposition~\ref{lem-s5.87-2}.
\item Type $B_\infty$ is not realizable in our setup, see
Proposition~\ref{prop-s6.2-1}.
\item Type $C_\infty$ admits a realization via projective-injective
modules in the BGG category $\mathcal{O}$ for $\mathfrak{sl}_2$, 
see Proposition~\ref{prop-s7.7-1}. We note that 
any realization of this type must have a non semi-simple
underlying category.
\item Type $D_\infty$ admits a realization via 
$\mathfrak{sl}_2$-Harish-Chandra modules over 
the semi-direct product 
$\mathfrak{g}=\mathfrak{sl}_2\ltimes L(4)$, 
see Proposition~\ref{prop-s8.7-1}. We also 
show that every simple transitive realization 
of this type forces the underlying category to 
be semi-simple, see Proposition~\ref{prop-s8.8-1}.
\item Type $T_\infty$ admits a realization via 
Whittaker $\mathfrak{sl}_2$-modules and
also via $\mathfrak{sl}_2$-Harish-Chandra modules over 
the Schr{\"o}dinger Lie algebra, 
see Proposition~\ref{prop-s9.7-1}. We also 
show that every simple transitive realization 
of this type forces the underlying category to 
be semi-simple, see Proposition~\ref{prop-s9.8-1}.
\end{itemize}

Several of our realizations are taken from the literature,
especially from \cite{MMr,MMr2}.
Additionally, we describe all $\mathscr{C}$-module categories 
inside $\mathfrak{sl}_2$-mod generated by simple modules
and also the $\mathscr{C}$-module categories
coming from the natural action of $\mathscr{C}$
on the categories of finite dimensional modules
over Lie subalgebras of $\mathfrak{sl}_2$.

Our non-semi-simple simple transitive realization of type
$A_\infty^\infty$ is especially interesting. This is because
applying $\mathscr{C}$ to simple objects in its 
abelianization never reaches non-zero projective objects.
This is in a very sharp contrast to representations of
finitary rigid monoidal categories, cf. 
\cite[Lemma~12]{MM5} and \cite[Theorem~2]{KMMZ}.
In fact, as far as we know, this is the first example
of such a phenomenon.

\subsection{Structure of the paper}\label{s1.3}

The paper is organized as follows:
Section~\ref{s2} contains all necessary preliminaries.
In Section~\ref{s3} we recall the classifications,
characterizations and combinatorics of Dynkin diagrams, 
both classical, affine and infinite.
Section~\ref{s4} studies realizations of type $A_\infty$.
Section~\ref{s5} studies realizations of type $A_\infty^\infty$.
Section~\ref{s7} studies realizations of type $C_\infty$.
Section~\ref{s8} studies realizations of type $D_\infty$.
Section~\ref{s9} studies realizations of type $T_\infty$.
Section~\ref{s6} discusses realizations of type $B_\infty$.
Finally, Section~\ref{s10} describes the
$\mathscr{C}$-module categories 
inside $\mathfrak{sl}_2$-mod generated by simple modules
and also the $\mathscr{C}$-module categories
coming from the natural action of $\mathscr{C}$
on the categories of finite dimensional modules
over Lie subalgebras of $\mathfrak{sl}_2$.
\vspace{5mm}

{\bf Acknowledgement.}

The first author is partially supported by the Swedish
Research Council.
\vspace{5mm}

\section{Monoidal category of finite dimensional 
$\mathfrak{sl}_2$-modules and its birepresentations}\label{s2}

\subsection{Preliminaries}\label{s2.1}

We work over the field $\mathbb{C}$ of complex numbers
and, as usual, write $\mathbb{C}^*$ for $\mathbb{C}\setminus\{0\}$.
For a Lie algebra $\mathfrak{g}$, we denote by 
$U(\mathfrak{g})$ the universal enveloping algebra
of $\mathfrak{g}$
and by $Z(\mathfrak{g})$ the center of $U(\mathfrak{g})$.
As usual, we denote by $\mathfrak{g}$-Mod the category of 
all $\mathfrak{g}$-modules and by $\mathfrak{g}$-mod 
the category of all finitely generated $\mathfrak{g}$-modules.

For all details on monoidal categories and their representations,
we refer to \cite{EGNO}.

\subsection{Finite dimensional 
$\mathfrak{sl}_2$-modules}\label{s2.2}

Consider the Lie algebra $\mathfrak{sl}_2$ with the 
standard basis $\{e,f,h\}$, where 
\begin{displaymath}
e=\left(\begin{array}{cc}0&1\\0&0\end{array}\right),\quad
f=\left(\begin{array}{cc}0&0\\1&0\end{array}\right),\quad
h=\left(\begin{array}{cc}1&0\\0&-1\end{array}\right).
\end{displaymath}
For details on the Lie algebra $\mathfrak{sl}_2$
and its modules we refer to \cite{Maz1}.

Denote by 
$\mathscr{C}$ the category of all finite dimensional 
$\mathfrak{sl}_2$-modules. The category $\mathscr{C}$
is semi-simple with simple objects $L(m)$, where $m\in\mathbb{Z}_{\geq 0}$.
The $\mathfrak{sl}_2$-module $L(m)$ is the unique simple 
$\mathfrak{sl}_2$-module of dimension $m+1$.  The module
$L(m)$ is a highest weight module of highest weight $m$.
Here a weight means just an $h$-eigenvalue.

The category $\mathscr{C}$ is monoidal with respect to the usual
tensor product of $\mathfrak{sl}_2$-modules. The monoidal unit is
the trivial $\mathfrak{sl}_2$-module $L(0)$. Recall the Clebsch-Gordan
coefficients for $\mathfrak{sl}_2$: for $m,n\in\mathbb{Z}_{\geq 0}$,
the $\mathfrak{sl}_2$-module $L(m)\otimes_\mathbb{C}L(n)$ decomposes as
\begin{displaymath}
L(|m-n|)\oplus L(|m-n|+2)\oplus L(|m-n|+4)\oplus \dots \oplus 
L(m+n-2)\oplus L(m+n).
\end{displaymath}
The monoidal category $\mathscr{C}$ is rigid and symmetric.
All objects of $\mathscr{C}$ are self-dual. The category 
$\mathscr{C}$ is generated by $L(1)$ as a monoidal category.
The category $\mathscr{C}$ is simple in the sense that it does
not have any non-trivial monoidal ideals.

To simplify notation, for every $i\in\mathbb{Z}_{\geq 0}$,
we will sometimes denote by $F_i$ the object $L(i)$ of $\mathscr{C}$.

\subsection{Locally finitary module categories over 
$\mathscr{C}$}\label{s2.3}

Let $\mathcal{M}$ be a (left) module category over $\mathscr{C}$.
We will say that $\mathcal{M}$ is {\em locally finitary}
provided that the following conditions are satisfied:
\begin{itemize}
\item $\mathcal{M}$ is $\mathbb{C}$-linear, additive, idempotent
split and Krull-Schmidt;
\item $\mathcal{M}$ has at most countably many indecomposable
objects, up to isomorphism;
\item all hom-spaces in $\mathcal{M}$ are finite dimensional
(over $\mathbb{C}$).
\end{itemize}
A typical example of a locally finitary 
module category over $\mathscr{C}$ is $\mathcal{M}=\mathscr{C}$
with respect to the left regular action.

Following \cite{MM5}, we will say that $\mathcal{M}$ is {\em transitive}
provided that, for any indecomposable objects $X,Y\in \mathcal{M}$,
there exists $F\in \mathscr{C}$ such that $Y$ is isomorphic to a summand
of $F(X)$. We will say that $\mathcal{M}$ is {\em simple transitive}
provided that $\mathcal{M}$ does not have any non-trivial 
$\mathscr{C}$-invariant ideals. 

Let $\{X_i\,:\,i\in I\}$ be a complete and irredundant set of 
representatives of the isomorphism classes of indecomposable 
objects in $\mathcal{M}$. Then, with any $F\in \mathscr{C}$, we can associate
an $I\times I$-matrix $[F]=(m_{i,j})_{i,j\in I}$ with 
non-negative integer coefficients in which $m_{i,j}$ is defined
as the multiplicity of $X_i$ as a summand of $F(X_j)$.
Then $\mathcal{M}$ is transitive if and only if, for all
$i,j\in I$, there is some power $[F_1]^k$ of $[F_1]$ such that
the $(i,j)$-th coefficient in $[F_1]^k$ is non-zero.

We also define the {\em action graph} $\Lambda_\mathcal{M}$ of 
$\mathcal{M}$ as an oriented graph (possibly with loops, but
without multiple oriented edges) with vertices $I$, in which 
there is an oriented edge from $i$ to $j$ if and only if
there is $F\in\mathscr{C}$ such that the coefficient 
$m_{j,i}$ in $[F]$ is non-zero. By taking as $F$ the unit object
in $\mathscr{C}$, we see that $\Lambda_\mathcal{M}$ has an oriented
loop at each vertex. Then $\mathcal{M}$ is transitive if and only if
$\Lambda_\mathcal{M}$ is strongly connected.

\subsection{Grothendieck ring of $\mathscr{C}$}\label{s2.4}

Consider the polynomial algebra $\mathbb{Z}[x]$. For
$i\geq 0$, define the polynomials $R_i(x)$ recursively as follows:
\begin{displaymath}
R_0(x)=1; \quad
R_1(x)=x; \quad
R_i(x)=xR_{i-1}(x)-R_{i-2}(x),
\text{ for } i\geq 2.
\end{displaymath}
These polynomials are sometimes called {\em ultraspherical polynomials},
see for example \cite{EO}. They can also be considered as renormalized
{\em Chebyshev polynomials of the second kind}, see \cite{Riv}.

The Grothendieck ring $[\mathscr{C}]$ of $\mathscr{C}$ is isomorphic
to $\mathbb{Z}[x]$ by sending the class of $L(i)$ in $[\mathscr{C}]$
to $R_i(x)$, for $i\geq 0$.
Consequently, for any locally finitary $\mathscr{C}$-module
category $\mathcal{M}$, we have $[F_i]=R_i([F_1])$, for $i\geq 0$.

Given a locally finitary $\mathscr{C}$-module category 
$\mathcal{M}$, the split Grothendieck group
$[\mathcal{M}]_\oplus$ of $\mathcal{M}$ has the obvious
structure of a $[\mathscr{C}]$-module. 

\subsection{Setup}\label{s2.5}

Let $\mathfrak{g}$ be a finite-dimensional Lie algebra 
with a fixed surjective Lie algebra homomorphism
$\varphi:\mathfrak{g}\to\mathfrak{sl}_2$. The trivial
example is $\mathfrak{g}=\mathfrak{sl}_2$ with
$\varphi$ being the identity. The pullback
along $\varphi$ allows us to consider 
$\mathscr{C}$ as  a (monoidal) subcategory of $\mathfrak{g}$-Mod,
up to monoidal equivalence.

We will be interested in $\mathscr{C}$-module categories of
the form $\mathcal{M}\subset \mathfrak{g}$-Mod in which the
$\mathscr{C}$-module structure is given by tensoring 
over $\mathbb{C}$ and then using the standard comultiplication
for $U(\mathfrak{g})$.  
We will be especially interested in 
such $\mathcal{M}$ given by the
additive closure of $\mathscr{C}\cdot M$, where 
$M\in \mathfrak{g}$-Mod, with particular emphasis on the case
when $M$ is simple. We will denote this kind of category by
$\mathrm{add}(\mathscr{C}\cdot M)$.

\subsection{Projective functors}\label{s2.6}

The category $\mathscr{C}$ is closely connected to 
the monoidal category of projective endofunctors 
for $\mathfrak{sl}_2$,
see \cite{BG}. The latter are endofunctors of the  full
subcategory $\mathcal{Z}$ of $\mathfrak{sl}_2$-Mod that consists of
all modules, the action of $Z(\mathfrak{sl}_2)$ on which is 
locally finite. Recall that $Z(\mathfrak{sl}_2)$ is 
a polynomial algebra generated by the {\em Casimir element}
$\mathtt{c}=(h+1)^2+4fe$. By \cite[Subsection~1.8]{BG}, 
the category $\mathcal{Z}$ is a direct sum, over all
$\theta\in\mathbb{C}$, of the full subcategories
$\mathcal{Z}_\theta$ of $\mathcal{Z}$ consisting of 
all modules on which $\mathtt{c}-\theta$ acts locally
nilpotently.

Let us recall classification of indecomposable projective functors
from \cite[Theorem~3.3]{BG}. For this, recall that $\mathfrak{sl}_2$
admits a triangular decomposition
\begin{displaymath}
\mathfrak{sl}_2=\mathfrak{n}_-\oplus \mathfrak{h}\oplus 
\mathfrak{n}_+, 
\end{displaymath}
where $e$ spans $\mathfrak{n}_+$, $f$ spans $\mathfrak{n}_-$
and $h$ spans $\mathfrak{h}$. We can identify 
$\mathfrak{h}^*$ with $\mathbb{C}$ by sending 
$\lambda\in \mathfrak{h}^*$ to $\lambda(h)$.
Then, for each $\lambda\in \mathbb{C}$, we have the corresponding
Verma module $\Delta(\lambda)$ with highest weight $\lambda$
and its unique simple quotient $L(\lambda)$.
Both these modules belong to the BGG category $\mathcal{O}$,
see \cite{BGG,Hu}. There, we also have the indecomposable 
projective cover $P(\lambda)$ of $L(\lambda)$. Note that
$P(\lambda)=\Delta(\lambda)$ unless $\lambda\in\{-2,-3,\dots\}$
and $P(\lambda)=L(\lambda)$ unless $\lambda\in(\mathbb{Z}\setminus\{-1\})$.

An element $\lambda\in \mathbb{C}$ is called {\em dominant}
provided that $\lambda\not\in\{-2,-3,\dots\}$ and {\em anti-dominant}
provided that $\lambda\not\in\mathbb{Z}_{\geq 0}$.
According to \cite[Theorem~3.3]{BG}, for each pair 
$(\lambda,\mu)\in\mathbb{C}^2$ satisfying
\begin{itemize}
\item $\lambda-\mu\in\mathbb{Z}$,
\item $\lambda$ is dominant,
\item $\mu$ is anti-dominant if $\lambda=-1$;
\end{itemize}
there is a unique indecomposable projective functor
\begin{displaymath}
\theta_{\lambda,\mu}:\mathcal{Z}_{(\lambda+1)^2}\to
\mathcal{Z}_{(\mu+1)^2}
\end{displaymath}
such that $\theta_{\lambda,\mu}\Delta(\lambda)=P(\mu)$.
All indecomposable projective endofunctors of 
$\mathcal{Z}$ are of this form.
If $\lambda\not\in\mathbb{Z}$, then $\theta_{\lambda,\mu}$
is an equivalence with inverse $\theta_{\mu,\lambda}$.

We will say that $\theta_{\lambda,\mu}$ is {\em special}
provided that $\lambda\in\frac{1}{2}+\mathbb{Z}$
and $\mu=-\lambda-2$.

\subsection{Abelianization}\label{s2.7}

Let $\mathcal{M}$ be a locally finitary 
$\mathscr{C}$-module category. Then we can
consider the projective abelianization 
$\overline{\mathcal{M}}$ of $\mathcal{M}$,
see \cite[Subsection~3.1]{MM1}. The objects of
$\overline{\mathcal{M}}$ are diagrams
$X\to Y$ over $\mathcal{M}$ and morphisms are 
(solid) commutative squares
\begin{displaymath}
\xymatrix{
X\ar[rr]\ar[d]&&Y\ar[d]\ar@{.>}[dll]\\X'\ar[rr]&&Y'
}
\end{displaymath}
modulo the ideal generated by all those squares
in which the right vertical morphism admits a
factorization via some dotted morphism.
The category
$\overline{\mathcal{M}}$ is a $\mathscr{C}$-module
category, with the $\mathscr{C}$-action inherited from
$\mathcal{M}$. 

If we additionally assume that, for each indecomposable
object $M\in \mathcal{M}$, there are only finitely 
many indecomposable objects $N$, up to isomorphism,
such that $\mathcal{M}(M,N)\neq 0$ and 
there are only finitely 
many indecomposable objects $N$, up to isomorphism,
such that $\mathcal{M}(N,M)\neq 0$, then
the category $\overline{\mathcal{M}}$
is an abelian length category in which the original
category $\mathcal{M}$ naturally embeds via
$P\mapsto (0\to P)$ and the closure of the
image of this embedding with respect to isomorphism
coincides with the category of projective objects in
$\overline{\mathcal{M}}$. 
In full generality, $\overline{\mathcal{M}}$ is abelian
with the above properties
if and only if $\mathcal{M}$ has weak kernels, see \cite{Fr}.
We will call such categories {\em admissible}.
We refer the reader to \cite{Mac1,Mac2} for more details
on locally finitary categories, where one can also find
some information about how some of the above conditions may be
relaxed.

\section{Generalized Cartan matrices and Dynkin diagrams}\label{s3}

\subsection{Generalized Cartan matrix}\label{s3.1}

Let $I$ be an indexing set (finite or infinite). Recall that 
a {\em generalized Cartan $I\times I$-matrix (GCM)} is a matrix
$C=(c_{i,j})_{i,j\in I}$ such that 
\begin{itemize}
\item $c_{i,i}\leq 2$, for all $i$;
\item $c_{i,j}\in \{0,-1,-2,\dots\}$, for all  $i\neq j$;
\item $c_{i,j}=0$ if and only if $c_{j,i}=0$, for all $i,j$.
\end{itemize}
Given a GCM $C$ on $I$, we can associate to it a
graph $\Gamma_C$ with vertex set $I$ such that,
for all $i\neq j\in I$, there are exactly $-c_{i,j}$ 
oriented edges from $i$ to $j$; moreover, for all $i\in I$, 
there are $2-c_{i,i}$ loops at the vertex $i$.
We denote by $\widetilde{C}$ the adjacency matrix of $\Gamma_C$
(with the convention that each loop corresponds to one on the main 
diagonal) and note that, by definition, $C=2\mathrm{Id}_I-\widetilde{C}$,
where $\mathrm{Id}_I$ is the identity $I\times I$-matrix.
A pair of oppositely oriented edges between $i$ and $j$,
for $i\neq j$, will be simplified to one unoriented edge
between these vertices. By convention, all loops are unoriented.

\begin{example}\label{ex-s3.1-1}
The graph associated to the GCM 
\begin{displaymath}
\left(\begin{array}{cccc}1&-1&0&0\\-2&2&-2&-1\\0&-1&2&-2\\0&-1&-2&2\end{array}\right) 
\end{displaymath}
looks as follows:

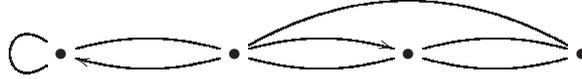
\begin{displaymath}
\xymatrix{
\bullet\ar@{-}@(ul,dl)[]\ar@/^2mm/@{-}[rr]&&
\bullet\ar@/^2mm/[rr]\ar@/_2mm/@{-}[rr]\ar@/^7mm/@{-}[rrrr]\ar@/^2mm/[ll]&&
\bullet\ar@/^2mm/@{-}[rr]\ar@/_2mm/@{-}[rr]&&
\bullet\\
}
\end{displaymath}
\end{example}

\subsection{Classical Dynkin diagrams and their characterization}\label{s3.2}

Recall the classical finite Dynkin diagrams, drawn with our convention for the
underlying graph (in all cases the index is the number of vertices):

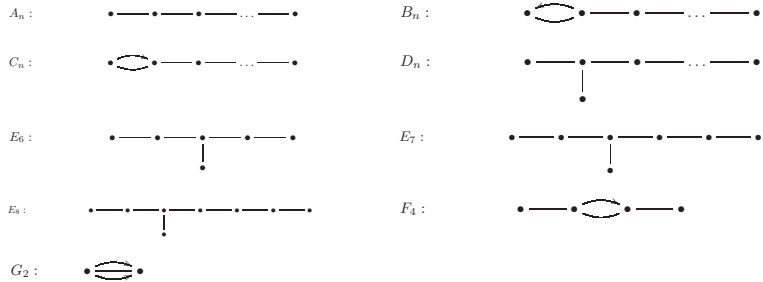
\resizebox{5cm}{!}{
$
\xymatrix{
A_n:&&\bullet\ar@{-}[r]&\bullet\ar@{-}[r]&\bullet\ar@{-}[r]&\dots\ar@{-}[r]&\bullet
}\qquad\qquad\qquad
$
}
\resizebox{5cm}{!}{
$
\xymatrix{
B_n:&&\bullet\ar@/_2mm/@{-}[r]&\bullet\ar@{-}[r]\ar@/_2mm/[l]&\bullet\ar@{-}[r]&\dots\ar@{-}[r]&\bullet
}
$
}

\resizebox{5cm}{!}{
$
\xymatrix{
C_n:&&\bullet\ar@/_2mm/@{-}[r]\ar@/^2mm/[r]&\bullet\ar@{-}[r]&\bullet\ar@{-}[r]&\dots\ar@{-}[r]&\bullet
}\qquad\qquad\qquad
$
}
\resizebox{5cm}{!}{
$
\xymatrix@R=4mm{
D_n:&&\bullet\ar@{-}[r]&\bullet\ar@{-}[r]\ar@{-}[d]&\bullet\ar@{-}[r]&\dots\ar@{-}[r]&\bullet\\
&&&\bullet&&&
}
$
}

\resizebox{5cm}{!}{
$
\xymatrix@R=4mm{
E_6:&&\bullet\ar@{-}[r]&\bullet\ar@{-}[r]&\bullet\ar@{-}[r]\ar@{-}[d]&
\bullet\ar@{-}[r]&\bullet\\
&&&&\bullet&&
}\qquad\qquad\qquad
$
}
\resizebox{5cm}{!}{
$
\xymatrix@R=4mm{
E_7:&&\bullet\ar@{-}[r]&\bullet\ar@{-}[r]&\bullet\ar@{-}[r]\ar@{-}[d]&
\bullet\ar@{-}[r]&\bullet\ar@{-}[r]&\bullet\\
&&&&\bullet&&
}
$
}

\resizebox{5cm}{!}{
$
\xymatrix@R=4mm{
E_8:&&\bullet\ar@{-}[r]&\bullet\ar@{-}[r]&\bullet\ar@{-}[r]\ar@{-}[d]&
\bullet\ar@{-}[r]&\bullet\ar@{-}[r]&\bullet\ar@{-}[r]&\bullet\\
&&&&\bullet&&
}\qquad\qquad\qquad
$
}
\resizebox{4cm}{!}{
$
\xymatrix@R=4mm{
F_4:&&\bullet\ar@{-}[r]&\bullet\ar@{-}@/_2mm/[r]\ar@/^2mm/[r]&\bullet\ar@{-}[r]&\bullet
}
$
}

\resizebox{2cm}{!}{
$
\xymatrix@R=4mm{
G_2:&\bullet\ar@{-}[r]\ar@/_2mm/[r]\ar@/^2mm/[r]&\bullet
}
$
}

An important invariant of a Dynkin diagram is the so-called
{\em Coxeter number}, usually denote by $\mathtt{h}$, 
given by the following table:
\begin{displaymath}
\begin{array}{c||c|c|c|c|c|c|c|c|c}
\text{type:}&A_n&B_n&C_n&D_n&E_6&E_7&E_8&F_4&G_2\\
\hline 
\mathtt{h}:&n+1&2n&2n&2n-2&12&18&30&12&6
\end{array}
\end{displaymath}

We can now recall the following standard characterization of 
the classical Dynkin diagrams, see \cite[Chapter~4]{Kac}
and \cite[Corollary~3.9]{EO}.

\begin{proposition}\label{prop-s3.2-1}
Let $Q$ be an irreducible matrix with non-negative
integer coefficients and zero diagonal. Then $Q=\widetilde{C}$, for
some classical Dynkin diagram, if and only if
the GCM $2\mathrm{Id}-Q$ is positive definite.
In the latter case, we have  $R_{\mathtt{h}-1}(Q)=0$,
where $\mathtt{h}$ is the Coxeter number of the Dynkin diagram in question.
\end{proposition}

In the context of representations of monoidal categories,
classical Dynkin diagrams appeared, for example, 
in \cite{KO,KMMZ}.

\subsection{Affine Dynkin diagrams and their characterization}\label{s3.3}

Next, recall the affine Dynkin diagrams (also known as generalized Euclidean diagrams):

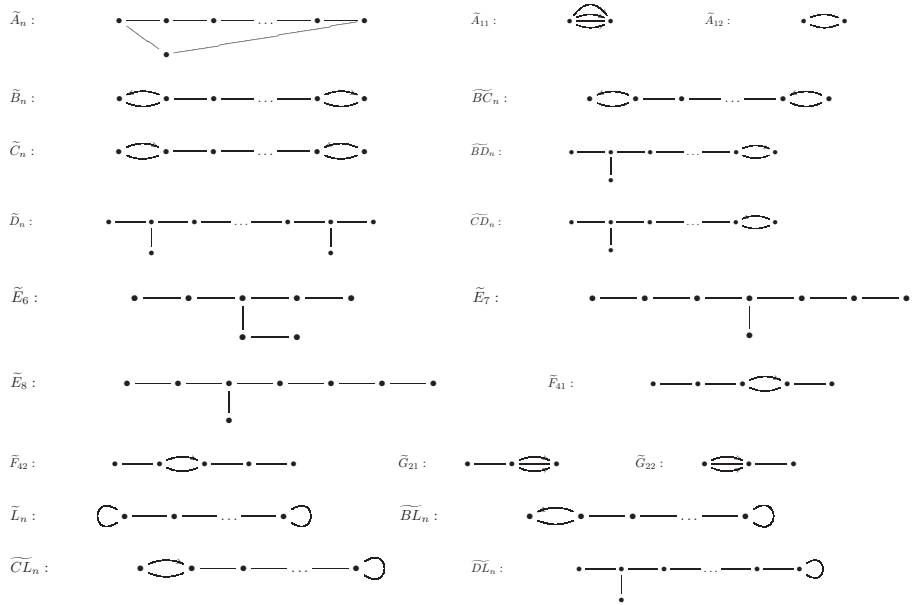
\resizebox{6cm}{!}{
$
\xymatrix@R=4mm{
\widetilde{A}_n:&&\bullet\ar@{-}[r]&\bullet\ar@{-}[r]&
\bullet\ar@{-}[r]&\dots\ar@{-}[r]&\bullet\ar@{-}[r]&\bullet\\
&&&\bullet\ar@{-}[lu]\ar@{-}[rrrru]&&&
}\qquad\qquad\qquad
$
}
\resizebox{3cm}{!}{
$
\xymatrix@R=4mm{
\widetilde{A}_{11}:&&
\bullet\ar@{-}[r]\ar@/_2mm/[r]\ar@/^2mm/[r]\ar@/^5mm/[r]&\bullet
}\qquad\qquad\qquad
$
}
\resizebox{3cm}{!}{
$
\xymatrix@R=4mm{
\widetilde{A}_{12}:&&\bullet\ar@/_2mm/@{-}[r]\ar@/^2mm/@{-}[r]&\bullet
}\qquad\qquad\qquad
$
}

\resizebox{6cm}{!}{
$
\xymatrix@R=4mm{
\widetilde{B}_n:&&\bullet\ar@/_2mm/@{-}[r]&
\bullet\ar@{-}[r]\ar@/_2mm/[l]&\bullet\ar@{-}[r]&\dots\ar@{-}[r]&
\bullet\ar@/_2mm/@{-}[r]\ar@/^2mm/[r]&\bullet
}\qquad\qquad\qquad
$
}
\resizebox{6cm}{!}{
$
\xymatrix@R=4mm{
\widetilde{BC}_n:&&\bullet\ar@/_2mm/@{-}[r]&
\bullet\ar@{-}[r]\ar@/_2mm/[l]&\bullet\ar@{-}[r]&
\dots\ar@{-}[r]&\bullet\ar@/_2mm/@{-}[r]&\bullet\ar@/_2mm/[l]
}\qquad\qquad\qquad
$
}

\resizebox{6cm}{!}{
$
\xymatrix@R=4mm{
\widetilde{C}_n:&&\bullet\ar@/_2mm/@{-}[r]\ar@/^2mm/[r]&
\bullet\ar@{-}[r]&\bullet\ar@{-}[r]&\dots\ar@{-}[r]&
\bullet\ar@/_2mm/@{-}[r]&\bullet\ar@/_2mm/[l]
}\qquad\qquad\qquad
$
}
\resizebox{6cm}{!}{
$
\xymatrix@R=4mm{
\widetilde{BD}_n:&&\bullet\ar@{-}[r]&
\bullet\ar@{-}[r]\ar@{-}[d]&\bullet\ar@{-}[r]&
\dots\ar@{-}[r]&\bullet\ar@/_2mm/@{-}[r]\ar@/^2mm/[r]&\bullet\\
&&&\bullet&&&&&&
}\qquad\qquad\qquad
$
}

\resizebox{6cm}{!}{
$
\xymatrix@R=4mm{
\widetilde{D}_n:&&\bullet\ar@{-}[r]&
\bullet\ar@{-}[r]\ar@{-}[d]&\bullet\ar@{-}[r]&\dots\ar@{-}[r]&
\bullet\ar@{-}[r]&\bullet\ar@{-}[r]\ar@{-}[d]&\bullet\\
&&&\bullet&&&&\bullet&
}\qquad\qquad\qquad
$
}
\resizebox{6cm}{!}{
$
\xymatrix@R=4mm{
\widetilde{CD}_n:&&\bullet\ar@{-}[r]&
\bullet\ar@{-}[r]\ar@{-}[d]&\bullet\ar@{-}[r]&
\dots\ar@{-}[r]&\bullet\ar@/_2mm/@{-}[r]&\bullet\ar@/_2mm/[l]\\
&&&\bullet&&&&&&
}\qquad\qquad\qquad
$
}

\resizebox{6cm}{!}{
$
\xymatrix@R=4mm{
\widetilde{E}_6:&&\bullet\ar@{-}[r]&\bullet\ar@{-}[r]&\bullet\ar@{-}[r]\ar@{-}[d]&
\bullet\ar@{-}[r]&\bullet\\
&&&&\bullet\ar@{-}[r]&\bullet&
}\qquad\qquad\qquad
$
}
\resizebox{6cm}{!}{
$
\xymatrix@R=4mm{
\widetilde{E}_7:&&\bullet\ar@{-}[r]&\bullet\ar@{-}[r]&
\bullet\ar@{-}[r]&\bullet\ar@{-}[r]\ar@{-}[d]&
\bullet\ar@{-}[r]&\bullet\ar@{-}[r]&\bullet\\
&&&&&\bullet&&
}
$
}

\resizebox{7cm}{!}{
$
\xymatrix@R=4mm{
\widetilde{E}_8:&&\bullet\ar@{-}[r]&\bullet\ar@{-}[r]&\bullet\ar@{-}[r]\ar@{-}[d]&
\bullet\ar@{-}[r]&\bullet\ar@{-}[r]&\bullet\ar@{-}[r]&\bullet\\
&&&&\bullet&&
}\qquad\qquad\qquad
$
}
\resizebox{5cm}{!}{
$
\xymatrix@R=4mm{
\widetilde{F}_{41}:&&\bullet\ar@{-}[r]&
\bullet\ar@{-}[r]&\bullet\ar@{-}@/_2mm/[r]\ar@/^2mm/[r]&\bullet\ar@{-}[r]&\bullet
}\qquad\qquad\qquad
$
}

\resizebox{5cm}{!}{
$
\xymatrix@R=4mm{
\widetilde{F}_{42}:&&\bullet\ar@{-}[r]&\bullet\ar@{-}@/_2mm/[r]\ar@/^2mm/[r]&\bullet\ar@{-}[r]&\bullet\ar@{-}[r]&\bullet
}\qquad\qquad\qquad
$
}
\resizebox{3cm}{!}{
$
\xymatrix@R=4mm{
\widetilde{G}_{21}:&\bullet\ar@{-}[r]&
\bullet\ar@{-}[r]\ar@/_2mm/[r]\ar@/^2mm/[r]&
\bullet
}\qquad\qquad
$
}
\resizebox{3cm}{!}{
$
\xymatrix@R=4mm{
\widetilde{G}_{22}:&\bullet\ar@{-}[r]\ar@/_2mm/[r]\ar@/^2mm/[r]&
\bullet\ar@{-}[r]&\bullet
}\qquad\qquad
$
}

\resizebox{5cm}{!}{
$
\xymatrix@R=4mm{
\widetilde{L}_{n}:&&\bullet\ar@{-}[r]\ar@{-}@(ul,dl)[]&
\bullet\ar@{-}[r]&\dots\ar@{-}[r]&\bullet\ar@{-}@(ur,dr)[]
}\qquad\qquad\qquad
$
}
\resizebox{6cm}{!}{
$
\xymatrix@R=4mm{
\widetilde{BL}_n:&&\bullet\ar@/_2mm/@{-}[r]&
\bullet\ar@{-}[r]\ar@/_2mm/[l]&\bullet\ar@{-}[r]&\dots\ar@{-}[r]&
\bullet\ar@{-}@(ur,dr)[]
}\qquad\qquad\qquad
$
}

\resizebox{6cm}{!}{
$
\xymatrix@R=4mm{
\widetilde{CL}_n:&&\bullet\ar@/_2mm/@{-}[r]\ar@/^2mm/[r]&
\bullet\ar@{-}[r]&\bullet\ar@{-}[r]&\dots\ar@{-}[r]&\bullet\ar@{-}@(ur,dr)[]
}\qquad\qquad\qquad
$
}
\resizebox{6cm}{!}{
$
\xymatrix@R=4mm{
\widetilde{DL}_n:&&\bullet\ar@{-}[r]&
\bullet\ar@{-}[r]\ar@{-}[d]&\bullet\ar@{-}[r]&\dots\ar@{-}[r]&
\bullet\ar@{-}[r]&\bullet\ar@{-}@(ur,dr)\\
&&&\bullet&&&&&
}\qquad\qquad\qquad
$
}

The following proposition characterizes the above graphs,
see \cite[Theorem~2]{HPR}:

\begin{proposition}\label{prop-s3.3-1}
A connected GCM on a finite index set which annihilates 
some  vector with positive integer coefficients is the
GCM of one of the affine Dynkin diagrams. 
\end{proposition}

\subsection{Infinite Dynkin diagrams and their characterization}\label{s3.4}

Finally, recall the following infinite Dynkin diagrams:

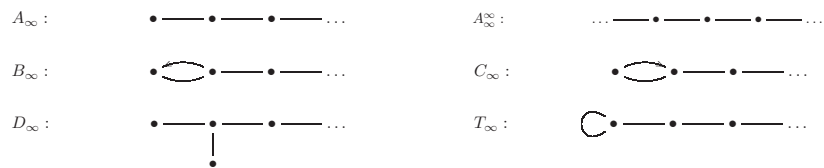
\resizebox{6cm}{!}{
$
\xymatrix@R=4mm{
A_\infty:&&\bullet\ar@{-}[r]&\bullet\ar@{-}[r]&
\bullet\ar@{-}[r]&\dots
}\qquad\qquad\qquad
$
}
\resizebox{6cm}{!}{
$
\xymatrix@R=4mm{
A_\infty^\infty:&&\dots\ar@{-}[r]&\bullet\ar@{-}[r]&\bullet\ar@{-}[r]&
\bullet\ar@{-}[r]&\dots
}\qquad\qquad\qquad
$
}

\resizebox{6cm}{!}{
$
\xymatrix@R=4mm{
B_\infty:&&\bullet\ar@/_2mm/@{-}[r]&\bullet\ar@{-}[r]\ar@/_2mm/[l]&
\bullet\ar@{-}[r]&\dots
}\qquad\qquad\qquad
$
}
\resizebox{6cm}{!}{
$
\xymatrix@R=4mm{
C_\infty:&&\bullet\ar@/_2mm/@{-}[r]\ar@/^2mm/[r]&\bullet\ar@{-}[r]&
\bullet\ar@{-}[r]&\dots
}\qquad\qquad\qquad
$
}

\resizebox{6cm}{!}{
$
\xymatrix@R=4mm{
D_\infty:&&\bullet\ar@{-}[r]&\bullet\ar@{-}[r]&
\bullet\ar@{-}[r]&\dots\\
&&&\bullet\ar@{-}[u]&
}\qquad\qquad\qquad
$
}
\resizebox{6cm}{!}{
$
\xymatrix@R=4mm{
T_\infty:&&\bullet\ar@{-}@(ul,dl)[]\ar@{-}[r]&\bullet\ar@{-}[r]&
\bullet\ar@{-}[r]&\dots
}\qquad\qquad\qquad
$
}

The following proposition characterizes the above graphs,
see \cite[Page~12]{HPR}:

\begin{proposition}\label{prop-s3.4-1}
A connected GCM on a countable index set which annihilates 
some  vector with positive integer coefficients is the
GCM of one of the infinite Dynkin diagrams.
\end{proposition}

We refer to \cite{HPR,HPR2} for more details.

\subsection{Combinatorics of transitive $\mathscr{C}$-categories}\label{s3.5}

The following corollary follows from the definitions and
the above characterizations of the affine and 
infinite Dynkin diagrams, see \cite[Theorem~4.1]{EK}.

\begin{corollary}\label{cor-s3.5-1}
Let $\mathcal{M}$ be a transitive $\mathscr{C}$-module
category for which $[L(1)]$ annihilates some vector with
positive integer coefficients. Then the action of $[L(1)]$ on the 
$[\mathscr{C}]$-module $[\mathcal{M}]_\oplus$ is 
given by the matrix $\widetilde{C}$ associated to
either an affine Dynkin diagram or an infinite Dynkin diagram.
\end{corollary}

This raises a natural problem to find a realization for each
possibility described in Corollary~\ref{cor-s3.5-1} within
the context provided by Subsection~\ref{s2.5}. This is the
main problem which we consider in this paper.
For the case of a finite rank, we have a (combinatorially
quite complex) classification
of simple $\mathscr{C}$-module categories obtained in
\cite{EO}. It is not clear where to look for such 
categories in our context. 
Therefore we focus on the infinite rank case.

\section{Realizations of type  $A_\infty$}\label{s4}

\subsection{First realization}\label{s4.1}

Let $\mathfrak{g}=\mathfrak{sl}_2$ with
$\varphi$ being the identity. Define
$\mathcal{M}_1:=\mathscr{C}$ with the left regular action.

Then the indecomposable objects in  $\mathcal{M}_1$
are $L(i)$, for $i\geq 0$, and the action of $L(1)$ is given by
\begin{displaymath}
L(1)\otimes_{\mathbb{C}} L(i)\cong
\begin{cases}
L(1),& i=0;\\
L(i-1)\oplus L(i+1),& i>0.
\end{cases}
\end{displaymath}
Therefore, we have
\begin{equation}\label{eq-s4.2-2}
[L(1)]=
\left(
\begin{array}{ccccc}
0&1&0&0&\dots\\
1&0&1&0&\dots\\
0&1&0&1&\dots\\
0&0&1&0&\dots\\
\vdots&\vdots&\vdots&\vdots&\ddots
\end{array}
\right)
\end{equation}

\subsection{Second realization}\label{s4.2}

Let $\mathfrak{g}=\mathfrak{sl}_2$ with
$\varphi$ being the identity. Consider the BGG category
$\mathcal{O}$ for $\mathfrak{g}$, see \cite{BGG,Maz1}.
For $\lambda\in\mathbb{C}$, denote by $L(\lambda)$
the simple highest weight $\mathfrak{g}$-module 
with highest weight $\lambda$. Also denote by
$P(\lambda)$ the projective cover of $L(\lambda)$
in $\mathcal{O}$.

Define the category $\mathcal{N}$ as the additive closure in
$\mathcal{O}$ of the modules $L(\lambda)$
and $P(\lambda)$, for $\lambda\in\{-1,-2,-3,\dots\}$.
Note that these are exactly the integral tilting
modules in the sense of \cite{CI,Rin}. Also, note
that $L(-1)=P(-1)$ while $L(\lambda)\not\simeq P(\lambda)$,
for $\lambda\in\{-2,-3,\dots\}$.

For $\lambda\in \{-1,-2,-3,\dots\}$, we have,
see \cite[Section~5]{Maz1}:
\begin{displaymath}
L(1)\otimes_{\mathbb{C}} L(\lambda)\cong
\begin{cases}
P(-2),& \lambda=-1;\\
L(\lambda-1)\oplus L(\lambda+1),& \lambda<-1;
\end{cases}
\end{displaymath}
and
\begin{equation}\label{eq-s4.2-3}
L(1)\otimes_{\mathbb{C}} P(\lambda)\cong
\begin{cases}
P(-2),& \lambda=-1;\\
P(-1)\oplus P(-1)\oplus P(-3),& \lambda=-2;\\
P(\lambda-1)\oplus P(\lambda+1),& \lambda<-2.
\end{cases}
\end{equation}

From these formulae it follows
that the additive closure of the modules $P(\lambda)$,
with $\lambda\in\{-1,-2,\dots\}$, is a transitive 
$\mathscr{C}$-module subcategory of $\mathcal{N}$.
Let $\mathcal{I}$ be the ideal in $\mathcal{N}$
generated by all $P(\lambda)$, where 
$\lambda\in\{-1,-2,\dots\}$. Then 
the formulae above imply that $\mathcal{I}$
is $\mathscr{C}$-stable and hence
$\mathcal{M}_2:=\mathcal{N}/\mathcal{I}$
is a $\mathscr{C}$-module category.
Now the above formulae imply Formula~\eqref{eq-s4.2-2}
for $\mathcal{M}_2$.

\subsection{Statements}\label{s4.3}

We can now summarize the content of this section
in the following proposition:

\begin{proposition}\label{prop-s4.2-1}
Both  $\mathcal{M}_1$ and $\mathcal{M}_2$ are simple
transitive $\mathscr{C}$-module categories of type $A_\infty$.
\end{proposition}

\begin{proof}
That  $\mathcal{M}_1$ and $\mathcal{M}_2$ are of 
type $A_\infty$, follows from Formula~\eqref{eq-s4.2-2}.
That both $\mathcal{M}_1$ and $\mathcal{M}_2$ are
simple transitive follows by combining that they are
transitive and semi-simple. The latter is obvious for
$\mathcal{M}_1$ and follows from Schur's lemma 
for $\mathcal{M}_2$, since $\mathcal{M}_2$ is, essentially, the 
additive closure of simple $\mathfrak{g}$-modules.
\end{proof}

We also record the following observation.

\begin{proposition}\label{prop-s4.2-2}
All admissible simple transitive $\mathscr{C}$-module categories 
of type $A_\infty$ are equivalent (as $\mathscr{C}$-module categories).
\end{proposition}

\begin{proof}
Let $\mathcal{M}$ be a simple transitive 
$\mathscr{C}$-module category of type $A_\infty$.
We are going to prove that $\mathcal{M}$ is equivalent,
as a $\mathscr{C}$-module category, to the left regular
$\mathscr{C}$-module category ${}_{\mathscr{C}}\mathscr{C}$.

We have that  $[L(1)]$ is given by Formula~\eqref{eq-s4.2-2}
and hence there is a unique, up to isomorphism 
indecomposable object $X$ in $\mathcal{M}$
such that $L(1)\cdot X$ is indecomposable. 

Since $L(i)\otimes_{\mathscr{C}} L(0)\cong L(i)$
in ${}_{\mathscr{C}}\mathscr{C}$, the first column of  
each $R_i([L(1)])$ has only one non-zero entry,
namely $1$ in row $i+1$. This means that 
$L(i)\cdot X$ is indecomposable, that 
$L(i)\cdot X\cong L(j)\cdot X$ if and only if $i=j$
and that all indecomposable objects of $\mathcal{M}$
can be obtained in this way.

Now let us show that the radical of $\mathcal{M}$ is
$\mathscr{C}$-invariant. Since $\mathscr{C}$ is generated
by $F_1$, it is enough to show that the radical of  $\mathcal{M}$ is
$F_1$-invariant. Let $M_0,M_1,\dots$ be a list of 
indecomposables in $\mathcal{M}$ such that in the 
corresponding basis of the split Grothendieck group
the action of $F_1$ is given by Formula~\eqref{eq-s4.2-2}.

If $\varphi: M_i\to M_j$ is a morphism and
$|i-j|\neq 0,2$, then $F_1(M_i)$ and $F_1(M_j)$ do not have
isomorphic summands and hence $F_1(\varphi)$
is a radical morphism. If $\varphi: M_i\to M_i$ 
is a radical morphism, it is nilpotent as
$\mathcal{M}$ is locally finitary. Therefore
$F_1(\varphi)$ is also nilpotent. As
$F_1(M_i)$ does not have isomorphic summands,
it follows that $F_1(\varphi)$ is a radical morphism.

It remains to consider the case of a morphism 
$\varphi: M_i\to M_{i\pm 2}$. Let us assume that
$\varphi: M_i\to M_{i+2}$, in the other case the
arguments are similar. Consider the abelianization
$\overline{\mathcal{M}}$ and let $N_i$
be the simple top of $M_i$ in $\overline{\mathcal{M}}$,
for $i\in\mathbb{Z}_{\geq 0}$.

\begin{lemma}\label{lem-s4.2-3}
For $i,j\in \mathbb{Z}_{\geq 0}$, we have
\begin{displaymath}
F_i(N_j)\cong N_{i+j}\oplus  N_{i+j-2}\oplus\dots
\oplus N_{|i-j|}.
\end{displaymath} 
\end{lemma}

\begin{proof}
Due to \eqref{eq-s4.2-2}, the matrices 
$[F_2]$, $[F_3]$, and so on, are as follows:
\begin{displaymath}
[F_2]=
\left(
\begin{array}{cccccc}
0&0&1&0&0&\dots\\
0&1&0&1&0&\dots\\
1&0&1&0&1&\dots\\
0&1&0&1&0&\dots\\
0&0&1&0&1&\dots\\
\vdots&\vdots&\vdots&\vdots&\ddots&\ddots
\end{array}
\right),\quad
[F_3]=
\left(
\begin{array}{cccccc}
0&0&0&1&0&\dots\\
0&0&1&0&1&\dots\\
0&1&0&1&0&\dots\\
1&0&1&0&1&\dots\\
0&1&0&1&0&\dots\\
\vdots&\vdots&\vdots&\vdots&\ddots&\ddots
\end{array}
\right),\dots
\end{displaymath}
In particular, from \cite[Lemma~8]{AM}
it follows that $F_i N_0\cong N_i$.
Denote by $\mathcal{N}$ the additive closure of 
all $N_i$. Then $\mathcal{N}$ is semi-simple
and has the structure of a $\mathscr{C}$-module 
category by restriction. By \cite[Lemma~8]{AM},
the matrix of the action of $F_1$ on 
$\mathcal{N}$ is given by \eqref{eq-s4.2-2}.
Now the statement of the Lemma follows from 
the matrix of $F_i$.
\end{proof}

Recall that we have a non-zero
morphism $\varphi: M_i\to M_{i+2}$.
Then $\varphi(M_i)$ is a submodule of the
radical of $M_{i+2}$. By Lemma~\ref{lem-s4.2-3}, 
we have $F_1(N_{i+2})\cong N_{i+1}\oplus N_{i+3}$.
This means that, applying the exact functor $F_1$
to the short exact sequence
\begin{displaymath}
0\to \mathrm{Rad}(M_{i+2}) \to
M_{i+2}\to N_{i+2}\to 0,
\end{displaymath}
we get a short exact sequence
\begin{displaymath}
0\to F_1(\mathrm{Rad}(M_{i+2})) \to
F_1(M_{i+2})\to F_1(N_{i+2})\to 0,
\end{displaymath}
in which $F_1(N_{i+2})$ is isomorphic to the top of $F_1(M_{i+2})$.
Therefore we have 
\begin{displaymath}
F_1(\mathrm{Rad}(M_{i+2}))=
\mathrm{Rad}(F_1(M_{i+2})).
\end{displaymath}
As $\varphi(M_i)$
is a submodule of $\mathrm{Rad}(M_{i+2})$
and $F_1$ is exact, we obtain that 
$F_1(\varphi(M_i))$ is a submodule of
$\mathrm{Rad}(F_1(M_{i+2}))$, which means 
that $F_1(\varphi)$ is a radical morphism.
This is exactly what
we needed to prove. This completes our argument that
the radical of $\mathcal{M}$ is $\mathscr{C}$-stable.

Now, the fact that $\mathcal{M}$ is simple transitive
implies that the radical of $\mathcal{M}$ is zero.
In other words, $\mathcal{M}$ is semi-simple.

By Yoneda Lemma, sending the unit object $L(0)$
to $X$ gives rise to a homomorphism of 
$\mathscr{C}$-module categories from 
${}_{\mathscr{C}}\mathscr{C}$ to
$\mathcal{M}$. From the previous paragraph,
this homomorphism is an equivalence.
\end{proof}

\section{Realizations of  type $A_\infty^\infty$}\label{s5}

\subsection{First realization}\label{s5.1}

Consider the setup of Subsection~\ref{s4.2}.
Let $\mathfrak{g}=\mathfrak{sl}_2$ with
$\varphi$ being the identity. Consider the BGG category
$\mathcal{O}$ for $\mathfrak{g}$.
For $\lambda\in\mathbb{C}\setminus\mathbb{Z}$,
consider the simple highest weight module
$L(\lambda)$ with highest weight $\lambda$.
Then the module $L(\lambda)$ is also projective in $\mathcal{O}$.

For a fixed $\mu\in \mathbb{C}\setminus\mathbb{Z}$,
let $\mathcal{N}_1=\mathcal{N}_1(\mu)$ be the additive closure of
all $L(\lambda)$ such that $\lambda-\mu\in\mathbb{Z}$.
Then, for any such $\lambda$, we have
\begin{displaymath}
L(1)\otimes_{\mathbb{C}}L(\lambda)\cong
L(\lambda-1)\oplus L(\lambda+1).
\end{displaymath}
In particular, $\mathcal{N}_1$ is a 
semi-simple and transitive $\mathscr{C}$-module
category.

From the above formula, we get that 
\begin{equation}\label{eq-s5.1-1}
[L(1)]=
\left(
\begin{array}{cccccc}
\ddots&\vdots&\vdots&\vdots&\vdots&\ddots\\
\dots&0&1&0&0&\dots\\
\dots&1&0&1&0&\dots\\
\dots&0&1&0&1&\dots\\
\dots&0&0&1&0&\dots\\
\ddots&\vdots&\vdots&\vdots&\vdots&\ddots
\end{array}
\right)
\end{equation}

\subsection{Second realization}\label{s5.2}

Let $\mathfrak{g}$ be the Schr{\"o}dinger Lie algebra
in the $1+1$-dimensional space-time,
see \cite{DLMZ}. This is a Lie algebra with basis
$\{e,f,h,p,q,z\}$, where $z$ is central and 
the rest of the Lie bracket is given by
\begin{gather*}
[h,e]=2e,\quad 
[h,f]=-2f,\quad 
[e,f]=h,\quad
[e,q]=p,\quad
[e,p]=0,\\
[h,p]=p,\quad
[f,p]=q,\quad
[f,q]=0,\quad
[h,q]=-q,\quad
[p,q]=z.
\end{gather*}
The subalgebra of $\mathfrak{g}$ generated by $e,f,h$
is isomorphic to $\mathfrak{sl}_2$ and the remaining
basis elements of $\mathfrak{g}$ span a 
nilpotent ideal. Thus, factoring that ideal out defines
a surjective Lie algebra homomorphism
$\varphi:\mathfrak{g}\to \mathfrak{sl}_2$.
Therefore we can view $\mathscr{C}$ as a monoidal 
subcategory of $\mathfrak{g}$-mod.

We have a natural triangular decomposition
$\mathfrak{g}=\mathfrak{n}_-\oplus \mathfrak{h}\oplus\mathfrak{n}_+$,
where $\mathfrak{n}_-$ is spanned by $f$ and $q$,
$\mathfrak{h}$ is spanned by $h$ and $z$, and
$\mathfrak{n}_+$ is spanned by $e$ and $p$.
Associated to this triangular decomposition,
we have the corresponding BGG category $\mathcal{O}$,
see \cite{DLMZ}. For $\lambda\in \mathfrak{h}^*$,
we denote by $L(\lambda)$ the simple highest weight
$\mathfrak{g}$-module with highest weight $\lambda$, this module 
is an object in $\mathcal{O}$.

Let $\lambda\in \mathfrak{h}^*$ be such that
$\lambda(z)\neq 0$ and $\lambda(h)\not\in\frac{1}{2}\mathbb{Z}$.
Then, due to \cite[Proposition~3]{DLMZ}, the module
$L(\lambda)$ is projective in $\mathcal{O}$.
Let $\varepsilon\in \mathfrak{h}^*$ be such that
$\varepsilon(z)=0$ and $\varepsilon(h)=1$.
By comparing the characters, we obtain that 
\begin{displaymath}
L(1)\otimes_\mathbb{Z}  L(\lambda)\cong
L(\lambda+\varepsilon)\oplus L(\lambda-\varepsilon).
\end{displaymath}

Now, for a fixed $\mu \in \mathfrak{h}^*$ such that $\mu(z)\neq 0$ 
and $\mu(h)\not\in\frac{1}{2}\mathbb{Z}$, define  
$\mathcal{N}_2$ as the additive closure of
all $L(\lambda)$ such that $\lambda-\mu\in\mathbb{Z}\varepsilon$.
Then the formula in the previous paragraph implies that
$\mathcal{N}_2$ is a semi-simple and transitive $\mathscr{C}$-module
category in which $[L(1)]$ is given by Formula~\eqref{eq-s5.1-1}.

\subsection{Third realization}\label{s5.3}

Let $\mathfrak{g}=\mathfrak{sl}_2\otimes_{\mathbb{C}}
\mathbb{C}[x]/(x^2)$ be the Takiff Lie algebra of
$\mathfrak{sl}_2$ and $\varphi:\mathfrak{g}\to\mathfrak{sl}_2$
be the natural projection. For $g\in \mathfrak{sl}_2$,
denote by  $\overline{g}$ the element $g\otimes x\in 
\mathfrak{g}$. Then the center 
$Z(\mathfrak{g})$ is generated by the elements
\begin{displaymath}
C=\overline{h}h +2\overline{e}f+2\overline{f}e\quad
\text{ and }\quad
\overline{C}=\overline{h}^2+4\overline{f}\overline{e}.
\end{displaymath}
The standard triangular decomposition of $\mathfrak{sl}_2$
induces a triangular decomposition of $\mathfrak{g}$
in the obvious way.
For $\lambda\in\mathbb{C}\setminus\mathbb{Z}$, consider 
the corresponding highest weight $\mathfrak{sl}_2$-module
$L(\lambda)$ with highest weight $\lambda$ and
set 
\begin{displaymath}
Q(\lambda):=U(\mathfrak{g})\bigotimes_{U(\mathfrak{sl}_2)}L(\lambda).
\end{displaymath}
Now,  for $\chi\in \mathbb{C}$, define
\begin{displaymath}
Q(\lambda,\chi):=Q(\lambda)/(\overline{C}-\chi)Q(\lambda)
\end{displaymath}
and, for $\theta\in \mathbb{C}$, define
\begin{displaymath}
Q(\lambda,\chi,\theta):=Q(\lambda,\chi)/(C-\theta)Q(\lambda,\chi). 
\end{displaymath}
Then, by \cite[Proposition~3.2]{Zh}, the module
$Q(\lambda,\chi,\theta)$ is a simple $\mathfrak{g}$-module
if and only if $\theta\neq \sqrt{\chi}(\lambda+2k)$, for
all $k\in\mathbb{Z}$. If the latter conditions are satisfied,
we have $Q(\lambda,\chi,\theta)\cong Q(\lambda',\chi',\theta')$
if and only if $\chi=\chi'$, $\theta=\theta'$ and
$\lambda-\lambda'\in 2\mathbb{Z}$.

Now let us fix $\lambda,\chi,\theta$ as above, satisfying
$\lambda\in\mathbb{C}\setminus\mathbb{Z}$ and
$\theta^2\neq \chi(\lambda+j)^2$, for
all $j\in\mathbb{Z}$, and, additionally, $\chi\neq 0$. 
Denote by $\mathcal{N}_3$ the additive
closure of $Q(\lambda+i,\chi,\theta+i\sqrt{\chi})$, for all
$i\in\mathbb{Z}$.

\begin{lemma}\label{lem-s5.3-1}
For $\lambda,\chi,\theta$ as above, the module
$L(1)\otimes_{\mathbb{C}}Q(\lambda,\chi,\theta)$
is isomorphic to the module
$Q(\lambda+1,\chi,\theta+\sqrt{\chi})\oplus 
Q(\lambda-1,\chi,\theta-\sqrt{\chi})$.
\end{lemma}

\begin{proof}
Note that $\overline{g}$ kills $L(1)$, for all $g\in \mathfrak{sl_2}$.
Therefore $\overline{C}$ acts on 
$L(1)\otimes_{\mathbb{C}}Q(\lambda,\chi,\theta)$
with the same scalar, namely $\chi$, as on $Q(\lambda,\chi,\theta)$.

Let us determine the action of $C$. By \cite[Lemma~3.1]{Zh},
the module $Q(\lambda,\chi,\theta)$, when restricted to
$\mathfrak{sl}_2$, is a multiplicity-free 
direct sum of the simple highest weight
modules $L(\lambda+i)$, where $i\in 2\mathbb{Z}$. As
$L(1)\otimes_{\mathbb{C}}L(\lambda+i)$ is isomorphic to
$L(\lambda+i+1)\oplus L(\lambda+i-1)$, the module
$L(1)\otimes_{\mathbb{C}}Q(\lambda,\chi,\theta)$
has two linearly independent elements of weight $\lambda+1$
that are killed by $e$. Let $\{b_1,b_{-1}\}$ be the standard basis
of $L(1)$, that is:
\begin{displaymath}
hb_1=b_1,\quad
hb_{-1}=-b_{-1},\quad
eb_{-1}=b_1,\quad
eb_1=0,\quad
fb_{-1}=0,\quad
fb_1=b_{-1}. 
\end{displaymath}
Let $v$
be a non-zero highest weight vector of $L(\lambda)\subset Q(\lambda,\chi,\theta)$.
Then we can take, as two linearly independent vectors in 
$L(1)\otimes_{\mathbb{C}}Q(\lambda,\chi,\theta)$
of weight $\lambda+1$ that are killed by $e$, the elements
\begin{displaymath}
u=b_1\otimes v \quad\text{ and }\quad
w=b_1\otimes (\overline{h}v)+2b_{-1}\otimes (\overline{e}v).
\end{displaymath}
It is easy to check that the matrix of the action of
$C$ on the linear span of $u$ and $w$ in the basis 
$\{u,w\}$ is given by
\begin{displaymath}
\left(\begin{array}{cc}\theta&\chi\\1&\theta\end{array}\right). 
\end{displaymath}
The eigenvalues of this matrix are $\theta\pm\sqrt{\chi}$
and they are different thanks to our assumption $\chi\neq 0$.
Now the claim of the lemma follows from the universal property
of the modules $Q(\lambda+1,\chi,\theta\pm\sqrt{\chi})$.
Note that here we use the simplicity of the modules
$Q(\lambda+1,\chi,\theta\pm\sqrt{\chi})$ which is
guaranteed by \cite[Proposition~3.2]{Zh} and
our assumption that $\theta^2\neq \chi(\lambda+j)^2$, for
all $j\in\mathbb{Z}$.
\end{proof}

Lemma~\ref{lem-s5.3-1} implies that
$\mathcal{N}_3$ is a semi-simple and transitive $\mathscr{C}$-module
category in which $[L(1)]$ is given by Formula~\eqref{eq-s5.1-1}.

\subsection{Fourth realization}\label{s5.4}

We stay in the setup of the previous subsection, that is, we
let $\mathfrak{g}=\mathfrak{sl}_2\otimes_{\mathbb{C}}
\mathbb{C}[x]/(x^2)$ be the Takiff Lie algebra of
$\mathfrak{sl}_2$ and $\varphi:\mathfrak{g}\to\mathfrak{sl}_2$
be the natural projection. Let $\chi\in\mathbb{C}^*$ and 
$\lambda\in \mathbb{C}^*$ be such that $\lambda^2=\chi$.
Then, \cite[Theorem~31]{MMr} describes a family of
simple $\mathfrak{g}$-module  $V(n,\lambda)$,
for $n\in\mathbb{Z}$, such that
\begin{itemize}
\item  $\overline{C}$ acts on $V(n,\lambda)$ as $\chi$;
\item when restricted to $\mathfrak{sl}_2$, the module $V(n,\lambda)$
is a direct sum of $L(|n|)$, $L(|n|+2)$ and so on;
\item $V(n,\lambda)=V(n',\lambda')$ if and only if 
$(n',\lambda')=(n,\lambda)$ or $(n',\lambda')=(-n,-\lambda)$.
\end{itemize}
By \cite[Proposition~29]{MMr}, we have 
\begin{displaymath}
L(1)\otimes_{\mathbb{C}}V(n,\lambda)\cong 
V(n-1,\lambda)\oplus V(n+1,\lambda).
\end{displaymath}

This implies that the additive closure $\mathcal{N}_4$ of all
$V(n,\lambda)$, for a fixed $\lambda$ and for all $n\in\mathbb{Z}$,
is a semi-simple and transitive $\mathscr{C}$-module
category in which $[L(1)]$ is given by Formula~\eqref{eq-s5.1-1}.

\subsection{Statements}\label{s5.7}

We can now summarize the content of this section
in the following proposition:

\begin{proposition}\label{prop-s5.7-1}
The categories  $\mathcal{N}_1$, $\mathcal{N}_2$, $\mathcal{N}_3$ 
and $\mathcal{N}_4$ are simple transitive 
$\mathscr{C}$-module categories of type $A_\infty^\infty$.
\end{proposition}

\begin{proof}
All these categories are simple
transitive as they are transitive and semi-simple.
The identification of the type  follows from Formula~\eqref{eq-s5.1-1}.
\end{proof}

\subsection{Comparison of $\mathcal{N}_1(\mu)$}\label{s5.8}

We can also compare the $\mathscr{C}$-module categories 
$\mathcal{N}_1(\mu)$ with each other, for varying values of $\mu$.
For the proofs, we need to recall the following notion:
for two $\mathfrak{sl}_2$-modules $M$ and $N$, we denote
by $\mathcal{L}(M,N)$ the 
$U(\mathfrak{sl}_2)$-$U(\mathfrak{sl}_2)$-subbimodule
of $\mathrm{Hom}_\mathbb{C}(M,N)$ consisting of 
all elements of the latter, the adjoint action of 
$\mathfrak{sl}_2$ on which is locally finite,
see \cite[Kapitel~6]{Ja} for details.

\begin{proposition}\label{prop-s5.8-1}
Let $\mu_1,\mu_2\in\mathbb{C}\setminus\mathbb{Z}$. Then
the $\mathscr{C}$-module categories $\mathcal{N}_1(\mu_1)$
and $\mathcal{N}_1(\mu_2)$ are equivalent, as
$\mathscr{C}$-module categories, if and only if 
$\mu_2\in \pm\mu_1+\mathbb{Z}$.
\end{proposition}

\begin{proof}
For $\lambda\in \mu_1+\mathbb{Z}$, consider
$\theta=(\lambda+1)^2$ and let $U_\theta$ denote the 
(primitive) quotient  of $U(\mathfrak{sl}_2)$ by the
ideal generated by $\mathtt{c}-\theta$. Due to our
choice of $\lambda$, we have that $L(\lambda)$
is a Verma module and hence we have an isomorphism
$\mathcal{L}(L(\lambda),L(\lambda))\cong U_\theta$
of $U(\mathfrak{sl}_2)$-$U(\mathfrak{sl}_2)$-bimodules,
see \cite[Subsection~6.9]{Ja}.

Further, by \cite[Subsection~6.8]{Ja}, we have an
isomorphism 
\begin{displaymath}
\mathrm{Hom}_{\mathfrak{sl}_2}(V,
\mathcal{L}(L(\lambda),L(\lambda))) \cong
\mathrm{Hom}_{\mathfrak{sl}_2}
(V\otimes_{\mathbb{C}}L(\lambda),L(\lambda)),
\end{displaymath}
where $\mathcal{L}(L(\lambda),L(\lambda))$
is considered as an $\mathfrak{sl}_2$-module under
the adjoint action. Comparing this 
with \cite[Section~7.9]{EGNO}, we obtain that algebra
$U_\theta$ is the algebra ind-object corresponding to 
the internal hom $\underline{\mathrm{Hom}}(L(\lambda),L(\lambda))$
for the $\mathscr{C}$-module category $\mathcal{N}_1(\mu_1)$.
By \cite[Theorem~7.10.1]{EGNO} and 
\cite[Remark~7.9.1]{EGNO}, the latter module category can be 
recovered as the category of certain $U_\theta$-modules in
the ind-completion of $\mathscr{C}$.

This implies that 
the $\mathscr{C}$-module categories $\mathcal{N}_1(\mu_1)$
and $\mathcal{N}_1(\mu_2)$ are equivalent, as
$\mathscr{C}$-module categories, if and only if 
there exists $\lambda_i\in\mu_i+\mathbb{Z}$, for
$i=1,2$, such that $U_{(\lambda_1+1)^2}\cong
U_{(\lambda_2+1)^2}$. By \cite[Theorem~6.4]{Di0},
we have $U_{(\lambda_1+1)^2}\cong
U_{(\lambda_2+1)^2}$ if and only if 
$(\lambda_1+1)^2=(\lambda_2+1)^2$. The claim follows.
\end{proof}

\subsection{Non-semi-simple realization}\label{s5.87}

In this subsection we construct, slightly outside of our
setup, a non-semi-simple simple transitive 
$\mathscr{C}$-module category of type $A_\infty^\infty$.

Consider the Lie algebra $\mathfrak{g}=\mathfrak{sl}_2$
and its Borel subalgebra $\mathfrak{b}$ generated by
$h$ and $e$. Denote by $\mathcal{N}$ the category of
all $\mathfrak{b}$-modules that are
\begin{itemize}
\item finitely generated;
\item $h$-diagonalizable;
\item free over $U(e)$.
\end{itemize}
For $\lambda\in\mathbb{C}$, we have a $\mathfrak{b}$-module
$N(\lambda)\in \mathcal{N}$ with basis 
$\{v_\mu\,:\,\mu\in\lambda+2\mathbb{Z}_{\geq0}\}$
and the $\mathfrak{b}$-action given by
\begin{displaymath}
hv_\mu=\mu v_\mu,\qquad ev_\mu=v_{\mu+2}. 
\end{displaymath}
The $\mathfrak{b}$-module $N(\lambda)$ is indecomposable
and is a free $U(e)$-module of rank one. Every object in 
$\mathcal{N}$ is a direct sum of finitely many objects of 
the form $N(\lambda)$. Note that $\mathcal{N}$ is not
semi-simple as $N(\lambda+2)\hookrightarrow N(\lambda)$,
for any $\lambda$.
For $\lambda\in\mathbb{C}$,
denote by $Q(\lambda)$ the quotient $N(\lambda)/N(\lambda+2)$.
Then $Q(\lambda)$ is a simple $\mathfrak{b}$-module and any 
simple finite dimensional $\mathfrak{b}$-module is of such form.

\begin{lemma}\label{lem-s5.87-1}
The category  $\mathcal{N}$ has the structure of a 
$\mathscr{C}$-module category by restriction.
\end{lemma}

\begin{proof}
Given $N\in \mathcal{N}$ and a finite dimensional 
$\mathfrak{sl}_2$-module $V$, we need to show that 
$(\mathrm{Res}^{\mathfrak{sl}_2}_{\mathfrak{b}}(V))\otimes_{\mathbb{C}}N$
belongs to $\mathcal{N}$. This module is, clearly,
$h$-diagonalizable and it is finitely generated as
$N$ is finitely generated and $V$ is finite dimensional.
The fact that this module is $U(e)$-free is proved similarly to the
well-known statement that tensoring with finite dimensional
modules in category $\mathcal{O}$ preserves the 
category of modules with Verma flag.
\end{proof}

For a fixed $ \lambda\in\mathbb{C}$, denote by 
$\mathcal{N}_5=\mathcal{N}_5(\lambda)$ the full 
additive subcategory of $\mathcal{N}$ generated by all
$N(\mu)$, for $\mu\in\lambda+\mathbb{Z}$. Clearly,
$\mathcal{N}_5$ inherits from $\mathcal{N}$ the structure of a
$\mathscr{C}$-module category by restriction. 

\begin{remark}\label{rem13}
{\em
Note that,
for any $\lambda,\lambda'\in\mathbb{C}$, the categories
$\mathcal{N}_5(\lambda)$ and $\mathcal{N}_5(\lambda')$
are equivalent as $\mathscr{C}$-module categories.
The equivalence from $\mathcal{N}_5(\lambda)$ to 
$\mathcal{N}_5(\lambda')$ is given by tensoring with $Q(\lambda'-\lambda)$
with the inverse given by tensoring with $Q(\lambda-\lambda')$.
In particular, the group
$\mathbb{Z}$ acts on $\mathcal{N}_5(\lambda)$ by
self-equivalences, where $k\in \mathbb{Z}$
maps $N(\lambda)$ to $N(\lambda+k)$.
}
\end{remark}

\begin{proposition}\label{lem-s5.87-2}
The  $\mathscr{C}$-module category $\mathcal{N}_5$ is simple
transitive of type $A_\infty^\infty$.
\end{proposition}

\begin{proof}
It is easy to check that 
\begin{displaymath}
(\mathrm{Res}^{\mathfrak{sl}_2}_{\mathfrak{b}}(L(1)))\otimes_{\mathbb{C}}N(\mu)
\cong N(\mu+1)\oplus N(\mu-1),
\end{displaymath}
from which it follows that $\mathcal{N}_5$ is a transitive 
$\mathscr{C}$-module category of type $A_\infty^\infty$.

Any radical morphism in $\mathcal{N}_5$ is a linear combination of some
morphisms of the form $\psi(\mu,k):=\big(N(\mu+2k)\hookrightarrow N(\mu)\big)$,
for $k\in\mathbb{Z}_{>0}$. Let $\{v_\nu\,:\, \nu\in\{\mu+2k,\mu+2k+2,\dots\}\}$
be the standard basis of $N(\mu+2k)$. Let 
$\{w_\nu\,:\, \nu\in\{\mu,\mu+2,\dots\}\}$
be the standard basis of $N(\mu)$. Let $\{e_{-1},e_1\}$ be the standard
basis of $L(1)$.

Applying $F_1$, we get an embedding
of the module $N(\mu+2k+1)\oplus N(\mu+2k-1)$ into 
$N(\mu+1)\oplus N(\mu-1)$. The summand $N(\mu+2k-1)$ is generated
by $e_{-1}\otimes v_{\mu+2k}$ and $F_1(\psi(\mu,k))$ maps the latter
element to $e_{-1}\otimes w_{\mu+2k}$.

The summand $N(\mu-1)$ is generated by $e_{-1}\otimes w_{\mu}$.
Applying $e^k$ to the latter, outputs the element
\begin{displaymath}
k e_1\otimes w_{\mu+2k-2}+ e_{-1}\otimes w_{\mu+2k}. 
\end{displaymath}
Therefore, projecting from $N(\mu+1)\oplus N(\mu-1)$ onto
$N(\mu+1)$ maps $e_{-1}\otimes w_{\mu+2k}$ to a non-zero element
of $N(\mu+1)$. In other words, the ideal generated by 
$F_1(\psi(\mu,k))$ contains $\psi(\mu+1,k-1)$. Proceeding
inductively, we eventually get the identity morphism 
on $N(\mu+k)$. This implies that the $\mathscr{C}$-stable ideal
generated by any non-zero morphism in $\mathcal{N}_5$ contains
the identity morphism on some non-zero objects. This means that
$\mathcal{N}_5$ is simple transitive and completes the proof.
\end{proof}

It is easy to check that, for any $i\in\mathbb{Z}_{\geq 0}$, the 
$\mathfrak{b}$-module $F_i(Q(\lambda))$ is isomorphic to the
quotient $N(\lambda-i)/N(\lambda+2+i)$, we denote the
latter module by $Q(\lambda,i)$. 
If we consider the additive closure $\mathcal{N}_6$
of all $Q(\lambda,i)$, for a fixed $\lambda$ and for
all $i\in\mathbb{Z}_{\geq 0}$, then it has the natural
structure of a $\mathscr{C}$-module category.

\begin{proposition}\label{lem-s5.87-3}
The  $\mathscr{C}$-module category $\mathcal{N}_6$ is 
simple transitive of type $A_\infty$.
\end{proposition}

\begin{proof}
It is straightforward to check that the 
$\mathfrak{b}$-module $F_1(Q(\lambda,i))$
is isomorphic to $Q(\lambda,i-1)\oplus Q(\lambda,i+1)$,
for $i>0$. This implies that $\mathcal{N}_6$ is 
transitive of type $A_\infty$. Simplicity 
is equivalent to the semi-simplicity of 
the underlying category of 
$\mathcal{N}_6$, that is, the equality
\begin{displaymath}
\mathrm{Hom}_\mathfrak{b}(Q(\lambda,i),
Q(\lambda,j))=\delta_{i,j}\mathbb{C}.
\end{displaymath}
Recall that $Q(\lambda,k)$ is a uniserial module 
with top $Q(\lambda-k)$, socle $Q(\lambda+k)$
and other composition subquotients 
$Q(\lambda-k+2)$, $Q(\lambda-4)$,\dots, 
$Q(\lambda+k-2)$.
In particular, since the top of $Q(\lambda,k)$
has multiplicity one, any endomorphism of $Q(\lambda,k)$
is scalar.

If $i>j$, then $[Q(\lambda,j):Q(\lambda-i)]=0$
and hence there are no non-zero homomorphisms from
$Q(\lambda,i)$ to $Q(\lambda,j)$ as the top of
$Q(\lambda,i)$ has nowhere to go.
If $i<j$, then we have $[Q(\lambda,i):Q(\lambda+j)]=0$
and hence there are no non-zero homomorphisms from
$Q(\lambda,i)$ to $Q(\lambda,j)$ as there is nothing
to cover the socle of $Q(\lambda,j)$.
This completes the proof.
\end{proof}

We can alternatively construct $\mathcal{N}_6$ by
first considering $\overline{\mathcal{N}_5}$, where
$Q(\lambda)$ is a simple object and then taking
the $\mathscr{C}$-module subcategory of 
$\overline{\mathcal{N}_5}$ generated by this object 
$Q(\lambda)$. 

The example of the categories $\mathcal{N}_5$
and $\mathcal{N}_6$
shows a striking difference between the
behaviour of $\mathscr{C}$ and that of
finitary monoidal categories
(more precisely, the fiat bi- and 2-categories
in the sense of \cite{MM1}).
In the latter case, starting from a simple 
object of the abelianization of a simple
transitive $2$-representation and applying
certain elements of $\mathscr{C}$ will output
projective objects, see \cite[Lemma~12]{MM5}
and, more generally, \cite[Theorem~2]{KMMZ}.
In our case, none of the $Q(\lambda,i)$
is projective in $\overline{\mathcal{N}_5}$.
The major contributor to this behaviour is the
failure of \cite[Proposition~18]{KM16} in the 
non-finitary context.

The $\mathscr{C}$-module category $\mathcal{N}_6$
does not depend on $\lambda$, up to equivalence.

\section{Realization of type  $C_\infty$}\label{s7}

\subsection{First realization}\label{s7.1}

Consider the setup of Subsection~\ref{s4.2}.
Let $\mathfrak{g}=\mathfrak{sl}_2$ with
$\varphi$ being the identity. Consider the BGG category
$\mathcal{O}$ for $\mathfrak{g}$ and the 
projective-injective modules $P(\lambda)\in\mathcal{O}$,
for $\lambda\in\{-1,-2,\dots\}$.
Let $\mathcal{K}_1$ be the additive closure of these
modules. Then $\mathcal{K}_1$ is exactly the category of 
integral projective injective modules in $\mathcal{O}$.
Consequently, it is closed with respect to tensoring with
finite dimensional $\mathfrak{sl}_2$-modules and hence
is a $\mathscr{C}$-module category.

From Formulae~\eqref{eq-s4.2-3}, we have
\begin{equation}\label{eq-s7.1-1}
[L(1)]=
\left(
\begin{array}{ccccc}
0&2&0&0&\dots\\
1&0&1&0&\dots\\
0&1&0&1&\dots\\
0&0&1&0&\dots\\
\vdots&\vdots&\vdots&\vdots&\ddots
\end{array}
\right)
\end{equation}

\subsection{Second realization}\label{s7.2}

Let $\mathfrak{g}=\mathfrak{sl}_2$ with
$\varphi$ being the identity. 
For a fixed non-zero $\xi\in\mathbb{C}$, let 
$M(\xi)$ be the $\mathfrak{g}$-module defined as the quotient of
$U(\mathfrak{g})$ by the left ideal generated by
$\mathtt{c}$ and $e-\xi$. Then, the module $M(\xi)$ 
is a simple $\mathfrak{g}$-module that appeared first 
in \cite{AP}. The module $M(\xi)$ is a Whittaker module in
the sense of \cite{Ko}. 

Set $X_{-1}:=M(\xi)$ and, for $i\in\{1,2,\dots\}$,
define $X_{-1-i}$ as the kernel of the central
element $(\mathtt{c}+i^2)^2$ acting on 
$L(i)\otimes_{\mathbb{C}}M(\xi)$. Then,
by \cite[Theorem~5.1]{MiSo}
(see also \cite[Theorem~67]{MS1} 
and \cite[Corollary~20]{MMM})
all these $X_{-1-i}$ are indecomposable 
(in fact, have length $2$ with isomorphic subquotients)
and the  additive closure $\mathcal{K}_2=\mathcal{K}_2(\xi)$ of all 
$X_j$, with $j\in \{-1,-2,\dots\}$, is closed 
under tensoring with finite dimensional
$\mathfrak{sl}_2$-modules.

\subsection{Statements}\label{s7.7}
We can now summarize the content of this section
in the following proposition:

\begin{proposition}\label{prop-s7.7-1}
Both  $\mathcal{K}_1$ and $\mathcal{K}_2$ are  simple
transitive $\mathscr{C}$-module categories of type $C_\infty$.
\end{proposition}

\begin{proof}
That  $\mathcal{K}_1$ has type $C_\infty$, 
follows from Formula~\eqref{eq-s7.1-1}.
If $\mathcal{K}_1$ is not simple transitive,
then there exists a non-trivial endomorphism $\psi$
of some $P(\lambda)$, for $\lambda\in\{-2,-3,\dots\}$,
such that  the $\mathscr{C}$-stable ideal $\mathcal{I}$ generated
by $\psi$ does not contain any identity morphism
of a non-zero object in $\mathcal{K}_1$.

If $\lambda=-2$, then, tensoring with $L(1)$ and
taking the kernel of $\mathtt{c}$ gives a nilpotent,
but non-zero endomorphism of $L(-1)\oplus L(-1)$.
As $\mathrm{add}(L(-1))$ is semi-simple, it follows
that $\mathcal{I}$ contains the identity on $L(-1)$,
a contradiction.

If $\lambda<-2$, then, tensoring with $L(1)$ and
taking the kernel of $(\mathtt{c}-(\lambda+2)^2)^2$
maps $\psi$ to a non-zero endomorphism of $P(\lambda+1)$
which must thus belong to $\mathcal{I}$. Proceeding
inductively, we will eventually come to the 
situation described in the previous paragraph.
This proves that $\mathcal{K}_1$ is simple 
transitive.

For $\mathcal{K}_2$, the claim follows from 
the claim for $\mathcal{K}_1$ and the equivalence
given by \cite[Theorem~5.1]{MiSo}.
\end{proof}

\begin{remark}\label{rem-s7.7-2}
{\em 
The category $\mathcal{K}_1$ is not semi-simple 
as the endomorphism algebras of the projective modules
$P(\lambda)$ are non-trivial, if $\lambda\in\{-2,-3,\dots\}$
(in fact, all these endomorphism algebras are isomorphic
to the dual numbers). This gives another example of a
non-semi-simple simple transitive module category over
a semi-simple rigid monoidal category. 

In fact, semi-simple simple transitive 
$\mathscr{C}$-module categories of type $C_\infty$
(or $B_\infty$) do not exist.
Indeed, in each such category simple and projective
objects coincide and hence, by \cite[Lemma~8]{AM},
the matrix $[F_1]$ should be symmetric.
In type $B_\infty$, we will have a stronger negative result
in Proposition~\ref{prop-s6.2-1}.
}
\end{remark}

\begin{proposition}\label{prop-s7.7-3}
All categories $\mathcal{K}_1$ and $\mathcal{K}_2(\xi)$, 
for $\xi\neq 0$, 
are equivalent as $\mathscr{C}$-module categories.
\end{proposition}

\begin{proof}
That $\mathcal{K}_1$
is equivalent, as a $\mathscr{C}$-module category,
to $\mathcal{K}_2(\xi)$, for any $\xi\neq 0$, follows
from \cite[Theorem~5.1]{MiSo}.

Here is a more elementary argument that all
$\mathcal{K}_2(\xi)$ are equivalent.
For $c\in\mathbb{C}^*$, we have an automorphism
$\varphi_c$ of $\mathfrak{sl}_2$ which keeps
$h$, sends $e$ to $ce$ and also sends $f$ to $\frac{1}{c}f$.
Directly from the definitions, we see that twisting
$M(\xi)$ by this automorphism outputs 
$M(c\xi)$ (or $M(c^{-1}\xi)$, depending on the twisting
conventions). Since such twisting preserves $\mathscr{C}$,
it gives rise to  an equivalence between $\mathcal{K}_2(\xi)$
and $\mathcal{K}_2(c\xi)$. 
\end{proof}

Using the ideas of \cite[Theorem~5.1]{MiSo} and,
more generally, the results in \cite[Theorem~67]{MS1} 
and \cite[Corollary~20]{MMM}, we can generalize
the example of the $\mathscr{C}$-module category
$\mathcal{K}_2$ by taking as $M$ any simple 
$\mathfrak{g}$-module which satisfies the
condition $\mathtt{c}M=0$.

\section{Realization of  type $D_\infty$}\label{s8}

\subsection{The realization}\label{s8.1}

Consider the semi-direct product
$\mathfrak{g}=\mathfrak{sl}_2\ltimes L(4)$, where
$L(4)$ is an abelian ideal and let
$\varphi:\mathfrak{g}\tto \mathfrak{sl}_2$
be the natural projection, see \cite{MMr2}.
Let $\{v_i\,:\,i\in\{0,\pm 2,\pm 4\}\}$ be the
standard basis of $L(4)$, that is, we have
$h\cdot v_i=iv_i$, then  $e\cdot v_4=0$ and  
$e\cdot v_{4-2i}=(5-i)v_{6-2i}$, for $i=1,2,3,4$.
The elements 
\begin{gather*}
C_2:=v_0^2-3v_{-2}v_2+12 v_{-4}v_4
\qquad\text{ and }\\
C_3:=v_0^3-\frac{9}{2}v_{-2}v_0v_2
+\frac{27}{2}v_{-2}^2v_4
+\frac{27}{2}v_{-4}v_2^2
-36v_{-4}v_0v_4
\end{gather*}
are central in $U(\mathfrak{g})$.
In \cite[Theorem~68]{MMr}, see also \cite[Theorem~12]{MMr2},
for every $\mu\in\mathbb{C}^*$,
one can find a construction of  simple $\mathfrak{g}$-modules
$V'(0,\mu)$, $V'(2,\mu)$ and $V(n,\mu)$, for 
$n\in\mathbb{Z}_{\geq 1}$,
such that
\begin{itemize}
\item $C_2$ and $C_3$ act on each such module as $\mu^2$
and $\mu^3$, respectively;
\item the restriction of each of these modules to
$\mathfrak{sl}_2$ is a direct sum of simple finite
dimensional modules with finite multiplicities.
\end{itemize}
By \cite[Lemma~14]{MMr2}, see also \cite[Remark~75]{MMr}, we have:
\begin{gather*}
L(1)\otimes_{\mathbb{C}} V(n,\mu)\cong
V(n-1,\mu)\oplus V(n+1,\mu),\text{ if } n>1;\\
L(1)\otimes_{\mathbb{C}} V(1,\mu)\cong
V'(0,\mu)\oplus V'(2,\mu)\oplus V(2,\mu);\\
L(1)\otimes_{\mathbb{C}}V'(0,\mu)\cong
V(1,\mu);\\
L(1)\otimes_{\mathbb{C}}V'(2,\mu)\cong
V(1,\mu).
\end{gather*}

This means that, with respect to the basis given by
the order 
\begin{displaymath}
V'(0,\mu),V'(2,\mu),V(1,\mu),V(2,\mu),V(3,\mu),\dots,
\end{displaymath}
we have:
\begin{equation}\label{eq-s8.1-1}
[L(1)]=
\left(
\begin{array}{cccccc}
0&0&1&0&0&\dots\\
0&0&1&0&0&\dots\\
1&1&0&1&0&\dots\\
0&0&1&0&1&\dots\\
0&0&0&1&0&\dots\\
\vdots&\vdots&\vdots&\vdots&\ddots
\end{array}
\right)
\end{equation}

We denote by $\mathcal{Y}'$ the additive closure of all these
$V'(0,\mu)$, $V'(2,\mu)$, $V(1,\mu)$, $V(2,\mu)$, $V(3,\mu),\dots$.

\subsection{Statement}\label{s8.7}
We have the following proposition:

\begin{proposition}\label{prop-s8.7-1}
The category $\mathcal{Y}'$ is a simple
transitive $\mathscr{C}$-module category of type $D_\infty$.
\end{proposition}

\begin{proof}
Simple transitivity follows by combining transitivity 
of the action of $\mathscr{C}$ with the semi-simplicity
of the underlying category. The identification of the
type follows from Formula~\eqref{eq-s8.1-1}.
\end{proof}

\subsection{Additional properties}\label{s8.8}

\begin{proposition}\label{prop-s8.8-1}
Let $\mathcal{Y}$ be an admissible  simple transitive 
$\mathscr{C}$-module category of type $D_\infty$.
Then $\mathcal{Y}$ is a semi-simple category.
\end{proposition}

\begin{proof}
We need to prove that the radical of $\mathcal{Y}$
is $\mathscr{C}$-invariant. Consider the abelianization
$\overline{\mathcal{Y}}$. Let $Y_0,Y_1,\dots$ be the
list of indecomposables in $\mathcal{Y}$ such that 
$[F_1]$ is given by \eqref{eq-s8.1-1} in the corresponding
basis of the Grothendieck group. These are the indecomposable
projectives in $\overline{\mathcal{Y}}$. Let $S_0,S_1,\dots$
be the corresponding simples in $\overline{\mathcal{Y}}$.
Note that \eqref{eq-s8.1-1} is a symmetric matrix and therefore
the matrix of the action of each $F_i$ in the basis 
of projectives coincides with the matrix of the action 
of this $F_i$ in the basis of simples.

We thus have $F_1(S_0)\cong F_1(S_1)\cong S_2$.
The module $F_1(S_2)$ has length three with 
simple subquotients $S_0$, $S_1$ and $S_3$, all appearing with 
multiplicity one. For $i\in\{0,1\}$, by adjunction,  we have
\begin{displaymath}
\overline{\mathcal{Y}}(F_1(S_2),S_i)\cong 
\overline{\mathcal{Y}}(S_2,F_1(S_i))\cong\mathbb{C} 
\end{displaymath}
and
\begin{displaymath}
\overline{\mathcal{Y}}(S_i,F_1(S_2))\cong 
\overline{\mathcal{Y}}(F_1(S_i),S_2)\cong\mathbb{C}. 
\end{displaymath}
Thus $S_i$ appears both in the socle and in the top
of $F_1(S_2)$,
which implies that $F_1(S_2)\cong S_0\oplus S_1\oplus S_3$.

Next we see that $F_1(S_3)$ has length two with 
simple subquotients $S_2$ and $S_4$, both appearing with 
multiplicity one. By adjunction, 
we have
\begin{displaymath}
\overline{\mathcal{Y}}(F_1(S_3),S_2)\cong 
\overline{\mathcal{Y}}(S_3,F_1(S_2))\cong \mathbb{C} 
\end{displaymath}
and 
\begin{displaymath}
\overline{\mathcal{Y}}(S_2,F_1(S_3))\cong 
\overline{\mathcal{Y}}(F_1(S_2),S_3)\cong \mathbb{C}. 
\end{displaymath}
Therefore $S_2$ appears both in the socle and in the top
of $F_1(S_3)$ implying that we have
$F_1(S_3)\cong S_2\oplus S_4$.

Proceeding inductively, we obtain that
$F_1(S_i)\cong S_{i-1}\oplus S_{i+1}$, for all $i\geq 3$.
Let $\tilde{\mathcal{Y}}$ be the additive closure of all $S_i$,
where $i\geq 0$.
We have just shown that 
$\tilde{\mathcal{Y}}$ is invariant under
the action of $F_1$. Since $F_1$ generates
$\mathscr{C}$, it follows that $\tilde{\mathcal{Y}}$ is
$\mathscr{C}$-invariant. In other words, 
$\tilde{\mathcal{Y}}$ is a simple transitive 
$\mathscr{C}$-module category of type $D_\infty$.

The rest is similar to the proof of
Proposition~\ref{prop-s4.2-2}. We need to show that 
the radical of $\mathcal{Y}$ is $F_1$-invariant.
Let $\varphi:Y_i\to Y_i$ be a radical morphism,
for some $i$. Then $\varphi$ is nilpotent
since $\mathrm{End}_{\mathcal{Y}}(Y_i)$
is finite dimensional. Hence $F_1(\varphi)$
is also nilpotent. Since $F_1(Y_i)$ does not have
isomorphic summands, it follows that 
$F_1(\varphi)$ is a radical morphism.

Let $\varphi:Y_i\to Y_j$ be a non-zero morphism,
for $i\neq j$ (hence $\varphi$ is automatically radical). 
Then $\varphi(Y_i)$ belongs to the
radical of $Y_j$. Applying the exact functor $F_1$
to the short exact sequence
\begin{displaymath}
0\to \mathrm{Rad}(Y_j) \to
Y_j\to S_j\to 0,
\end{displaymath}
we get a short exact sequence
\begin{displaymath}
0\to F_1(\mathrm{Rad}(Y_j)) \to
F_1(Y_j)\to F_1(S_j)\to 0,
\end{displaymath}
in which $F_1(S_j)$ is isomorphic to the top of $F_1(Y_j)$
by the above.
Therefore we have 
\begin{displaymath}
F_1(\mathrm{Rad}(Y_{j}))=
\mathrm{Rad}(F_1(Y_j)).
\end{displaymath}
As $\varphi(Y_i)\subset \mathrm{Rad}(Y_{j})$
and $F_1$ is exact, we obtain that 
$F_1(\varphi(Y_i))$ is a submodule of
$\mathrm{Rad}(F_1(Y_j))$, which means 
that $F_1(\varphi)$ is a radical morphism.
As the radical of $\mathcal{Y}$ is generated
by radical morphisms between indecomposable objects,
it follows that the radical of $\mathcal{Y}$ 
is $\mathscr{C}$-stable and completes the proof.
\end{proof}

\section{Realization of  type $T_\infty$}\label{s9}

\subsection{First realization}\label{s9.1}

Let $\mathfrak{g}=\mathfrak{sl}_2$ with
$\varphi$ being the identity. 
For a fixed non-zero $\xi\in\mathbb{C}^*$
and $\lambda\in\mathbb{Z}+\frac{1}{2}$, let 
$M(\lambda,\xi)$ be the $\mathfrak{g}$-module defined as the quotient of
$U(\mathfrak{g})$ by the left ideal generated by
$\mathtt{c}-(\lambda+1)^2$ and $e-\xi$. Then, the module $M(\lambda,\xi)$ 
is a simple Whittaker $\mathfrak{g}$-module, see \cite{AP,Ko}.
Note that $M(\lambda,\xi)\cong M(-\lambda-2,\xi)$ due to our choice
of $\lambda$.

As $\lambda$ is not an integer, the functor
$L(1)\otimes_{\mathbb{C}}{}_-$ goes from
$\mathcal{Z}_{(\lambda+1)^2}$ to 
$\mathcal{Z}_{\lambda^2}\oplus \mathcal{Z}_{(\lambda+2)^2}$
and is isomorphic to the direct sum of
$\theta_{\lambda,\lambda-1}$ and $\theta_{\lambda,\lambda+1}$.
It follows that $L(1)\otimes_{\mathbb{C}}M(\lambda,\xi)$
is a direct sum of $M(\lambda-1,\xi)$ and $M(\lambda+1,\xi)$.
Denote by $\mathcal{X}_1$ the additive closure of all
$M(\lambda+i,\xi)$, for $i\in\mathbb{Z}$. This coincides with
the additive closure of all $M(-\frac{3}{2}-i,\xi)$, for 
$i\in\mathbb{Z}_{\geq 0}$, the latter now being pair-wise non-isomorphic.

If $\lambda=-\frac{3}{2}$, then $M(\lambda+1,\xi)\cong M(\lambda,\xi)$.
If $\lambda=-\frac{1}{2}$, then $M(\lambda-1,\xi)\cong M(\lambda,\xi)$.
Therefore, we have
\begin{equation}\label{eq-s9.1-1}
[L(1)]=
\left(
\begin{array}{ccccc}
1&1&0&0&\dots\\
1&0&1&0&\dots\\
0&1&0&1&\dots\\
0&0&1&0&\dots\\
\vdots&\vdots&\vdots&\vdots&\ddots
\end{array}
\right)
\end{equation}

\subsection{Second realization}\label{s9.2}

Let $\mathfrak{g}$ be the Schr{\"o}dinger Lie algebra 
as in Subsection~\ref{s5.2}, see \cite{DLMZ},
and $\varphi:\mathfrak{g}\to \mathfrak{sl}_2$
be the natural projection. 
Then, for each $\theta\in\mathbb{C}^*$,
\cite[Theorem~54]{MMr} describes a family of
simple $\mathfrak{g}$-modules  $V(n,\theta)$,
where $n\in\mathbb{Z}$, such that
\begin{itemize}
\item  $z$ acts on $V(n,\theta)$ as $\chi$;
\item when restricted to $\mathfrak{sl}_2$, the module $V(n,\lambda)$
is a direct sum of $L(|n|)$, $L(|n|+1)$ and so on;
\item $V(n,\theta)=V(n',\theta')$ if and only if 
$(n',\theta')=(n,\theta)$ or $(n',\theta')=(-n,\theta)$.
\end{itemize}
By \cite[Proposition~53]{MMr}, we have 
\begin{displaymath}
L(1)\otimes_{\mathbb{C}}V(n,\theta)\cong 
\begin{cases}
V(n-1,\theta)\oplus V(n+1,\theta),& n\neq 0;\\
V(0,\theta)\oplus V(1,\theta),& n= 0.\\
\end{cases}
\end{displaymath}

For a fixed $\theta\in\mathbb{C}^*$, denote by 
$\mathcal{X}_2$ the additive closure of all
$V(n,\theta)$, for $n\in\mathbb{Z}$. By the above,
$\mathcal{X}_2$ is semi-simple and has a natural structure
of a $\mathscr{C}$-module category in which 
$[L(1)]$ is given by \eqref{eq-s9.1-1}.

\subsection{Statement}\label{s9.7}

We can now summarize the content of this section
in the following proposition:

\begin{proposition}\label{prop-s9.7-1}
The categories $\mathcal{X}_1$ and $\mathcal{X}_2$ are simple
transitive $\mathscr{C}$-module categories of type $T_\infty$.
\end{proposition}

\begin{proof}
Simple transitivity follows by combining transitivity 
of the action of $\mathscr{C}$ with the semi-simplicity
of the underlying categories. The identification of the
type follows from Formula~\eqref{eq-s9.1-1}.
\end{proof}

\subsection{Additional properties}\label{s9.8}

\begin{proposition}\label{prop-s9.8-1}
Let $\mathcal{X}$ be an admissible  simple transitive 
$\mathscr{C}$-module category of type $T_\infty$.
Then $\mathcal{X}$ is a semi-simple category.
\end{proposition}

\begin{proof}
We need to prove that the radical of $\mathcal{X}$
is $\mathscr{C}$-invariant. Consider the abelianization
$\overline{\mathcal{X}}$. Let $X_0,X_1,\dots$ be the
list of indecomposables in $\mathcal{X}$ such that 
$[F_1]$ is given by \eqref{eq-s9.1-1} in the corresponding
basis of the Grothendieck group. These are the indecomposable
projectives in $\overline{\mathcal{X}}$. Let $S_0,S_1,\dots$
be the corresponding simples in $\overline{\mathcal{X}}$.
Note that \eqref{eq-s9.1-1} is a symmetric matrix and therefore
the matrix of the action of each $F_i$ in the basis 
of projectives coincides with the matrix of the action 
of this $F_i$ in the basis of simples.

In particular, $F_1(S_0)$ has length two with 
simple subquotients $S_0$ and $S_1$, both appearing with 
multiplicity one. Then $S_0$ must appear either in the
top or in the socle of $F_1(S_0)$. However, by adjunction, 
we have
\begin{displaymath}
\overline{\mathcal{X}}(F_1(S_0),S_0)\cong 
\overline{\mathcal{X}}(S_0,F_1(S_0)) 
\end{displaymath}
and thus $S_0$ appears both in the socle and in the top,
which implies that $F_1(S_0)\cong S_0\oplus S_1$.

Next we see that $F_1(S_1)$ has length two with 
simple subquotients $S_0$ and $S_2$, both appearing with 
multiplicity one. By adjunction, 
we have
\begin{displaymath}
\overline{\mathcal{X}}(F_1(S_1),S_0)\cong 
\overline{\mathcal{X}}(S_1,F_1(S_0))\cong \mathbb{C} 
\end{displaymath}
and 
\begin{displaymath}
\overline{\mathcal{X}}(S_0,F_1(S_1))\cong 
\overline{\mathcal{X}}(F_1(S_0),S_1)\cong \mathbb{C}. 
\end{displaymath}
Therefore $S_0$ appears both in the socle and in the top
of $F_1(S_1)$ implying that we have
$F_1(S_1)\cong S_0\oplus S_2$.

Proceeding inductively, we obtain that
$F_1(S_i)\cong S_{i-1}\oplus S_{i+1}$, for all $i\geq 1$.
Let $\tilde{\mathcal{X}}$ be the additive closure of all $S_i$,
where $i\geq 0$.
We have just shown that 
$\tilde{\mathcal{X}}$ is invariant under
the action of $F_1$. Since $F_1$ generates
$\mathscr{C}$, it follows that $\tilde{\mathcal{X}}$ is
$\mathscr{C}$-invariant. In other words, 
$\tilde{\mathcal{X}}$ is a simple transitive 
$\mathscr{C}$-module category of type $T_\infty$.

The rest is similar to the proof of
Proposition~\ref{prop-s4.2-2}. We need to show that 
the radical of $\mathcal{X}$ is $F_1$-invariant.
Let $\varphi:X_i\to X_i$ be a radical morphism,
for some $i$. Then $\varphi$ is nilpotent
since $\mathrm{End}_{\mathcal{X}}(X_i)$
is finite dimensional. Hence $F_1(\varphi)$
is also nilpotent. Since $F_1(X_i)$ does not have
isomorphic summands, it follows that 
$F_1(\varphi)$ is a radical morphism.

Let $\varphi:X_i\to X_j$ be a non-zero morphism,
for $i\neq j$. Then $\varphi(X_i)$ belongs to the
radical of $X_j$. Applying the exact functor $F_1$
to the short exact sequence
\begin{displaymath}
0\to \mathrm{Rad}(X_j) \to
X_j\to S_j\to 0,
\end{displaymath}
we get a short exact sequence
\begin{displaymath}
0\to F_1(\mathrm{Rad}(X_j)) \to
F_1(X_j)\to F_1(S_j)\to 0,
\end{displaymath}
in which $F_1(S_j)$ is isomorphic to the top of $F_1(X_j)$
by the above.
Therefore we have 
\begin{displaymath}
F_1(\mathrm{Rad}(X_{j}))=
\mathrm{Rad}(F_1(X_j)).
\end{displaymath}
As $\varphi(X_i)\subset \mathrm{Rad}(X_{j})$
and $F_1$ is exact, we obtain that 
$F_1(\varphi(X_i))$ is a submodule of
$\mathrm{Rad}(F_1(X_j))$, which means 
that $F_1(\varphi)$ is a radical morphism.
As the radical of $\mathcal{X}$ is generated
by radical morphisms between indecomposable objects,
it follows that the radical of $\mathcal{X}$ 
is $\mathscr{C}$-stable and completes the proof.
\end{proof}

\section{Good and bad news on realizations of type $B_\infty$}\label{s6}

\subsection{A realizations dual to $C_\infty$}\label{s6.1}

The easiest way to find a $B_\infty$-combinatorics is,
of course, by taking the dual of a 
$C_\infty$-combinatorics. Let us do this with the
example from Subsection~\ref{s7.1}.

Let $\mathfrak{g}=\mathfrak{sl}_2$ with
$\varphi$ being the identity. Consider the BGG category
$\mathcal{O}$ for $\mathfrak{g}$ and the 
projective-injective modules $P(\lambda)\in\mathcal{O}$,
for $\lambda\in\{-1,-2,\dots\}$.
Let $\mathcal{K}_1$ be the additive closure of these
modules. Then $\mathcal{K}_1$ is a simple transitive
$\mathscr{C}$-module category of type $C_\infty$,
see Subsection~\ref{s7.1} and Proposition~\ref{prop-s7.7-1}.

Consider the projective abelianization 
$\overline{\mathcal{K}_1}$ of $\mathcal{K}_1$,
see Subsection~\ref{s2.7}. 

Each indecomposable projective object $P(\lambda)$
of $\overline{\mathcal{K}_1}$ has unique simple top,
which we denote by $L(\lambda)$. It can be identified
with the simple highest weight $\mathfrak{sl}_2$-module
with highest weight $\lambda$. Note that here
$\lambda$ is a negative integer. Recall that
\begin{displaymath}
[P(\lambda):L(\lambda')]=
\begin{cases}
2,& \lambda=\lambda'\in\{-2,-3,\dots\};\\
1,& \lambda=\lambda'=-1;\\
0,& \text{otherwise}.
\end{cases}
\end{displaymath}

Now, note that each element of $\mathscr{C}$ acts
as an exact endofunctor of $\overline{\mathcal{K}_1}$.
From \eqref{eq-s4.2-3} it follows that,
in the basis $\{[L(-1)],[L(-2)],\dots\}$
of the Grothendieck group of $\overline{\mathcal{K}_1}$,
the action of $L(1)$ is given by the following matrix:
\begin{equation}\label{eq-s6.1-1}
\left(
\begin{array}{ccccc}
0&1&0&0&\dots\\
2&0&1&0&\dots\\
0&1&0&1&\dots\\
0&0&1&0&\dots\\
\vdots&\vdots&\vdots&\vdots&\ddots
\end{array}
\right)
\end{equation}
which is a type $B_\infty$ matrix (i.e. the transpose
of the type $C_\infty$ matrix).

\subsection{Impossibility of realizations of type $B_\infty$
in our setup}\label{s6.2}

\begin{proposition}\label{prop-s6.2-1}
Locally finitary, admissible, simple  transitive $\mathscr{C}$-module 
categories  of type $B_\infty$ over the complex numbers do not exist.
\end{proposition}

\begin{proof}
Let $\mathcal{X}$ be a simple transitive $\mathscr{C}$-module
category of type $B_\infty$. Let $X_1,X_2,\dots$ be fixed representatives
of isomorphism classes  of   indecomposable objects in  $\mathcal{X}$
such that, with respect to them, the matrix $[L(1)]$  is given by
\eqref{eq-s6.1-1}. Consider the abelianization 
$\overline{\mathcal{X}}$ as in the previous subsection and,
for $i=1,2,\dots$, denote  by $N_i$   the simple top of $P_i:=(0\to X_i)$
in $\overline{\mathcal{X}}$. 

All elements of $\mathscr{C}$ operate on $\overline{\mathcal{X}}$
as exact self-adjoint functors. Therefore we can consider the matrix 
of each such element written in the basis of the Grothendieck  group of
$\overline{\mathcal{X}}$ given by simple objects. By
\cite[Lemma~8]{AM}, this matrix is just the transpose of the 
matrix in the basis of the indecomposable  projective objects.
Therefore the corresponding matrix for $L(1)$  is given  by
\eqref{eq-s7.1-1}, that is, it has type $C_\infty$.

Directly from
this matrix of $L(1)$, we see that $F_1(N_1)\cong N_2$.  Also,
$F_1(N_2)$ is a module of length three with composition
subquotients $N_1$, appearing with multiplicity $2$,  and
$N_3$,  appearing with multiplicity one. As $N_2$ is a simple
object and $\mathbb{C}$ is algebraically closed, we have
$\overline{\mathcal{X}}(N_2,N_2)=\mathbb{C}$. By adjunction,
we thus have
\begin{equation}\label{eq-s6.2-2}
\mathbb{C}= \overline{\mathcal{X}}(N_2,N_2)=
\overline{\mathcal{X}}(F_1(N_1),F_1(N_1))=
\overline{\mathcal{X}}(N_1,F_1\circ F_1(N_1)).
\end{equation}
Similarly, we also have
\begin{equation}\label{eq-s6.2-3}
\mathbb{C}= 
\overline{\mathcal{X}}(F_1\circ F_1(N_1),N_1).
\end{equation}

We have $F_1\circ F_1\cong F_0\oplus F_2$. Therefore
\begin{displaymath}
F_1\circ F_1(N_1)\cong
F_0(N_1)\oplus F_2(N_1)\cong
N_1\oplus F_2(N_1).
\end{displaymath}
The module $F_2(N_1)$ has length two with composition
subquotients $N_1$ and $N_3$. Therefore $N_1$ must appear
either in the top or in the socle of $F_2(N_1)$.
In the first case,  we have two copies of $N_1$ in the top of
$F_1\circ F_1(N_1)$ leading to a contradiction with \eqref{eq-s6.2-3},
In the second case,  we get a similar contradiction with \eqref{eq-s6.2-2}.
This proves the claim.
\end{proof}

Proposition~\ref{prop-s6.2-1} implies that, in order 
to have a chance to construct a locally finitary 
simple  transitive $\mathscr{C}$-module 
category of type $B_\infty$, we need to change the base field
to some field that is not algebraically closed, which would
allow the endomorphism algebra of a simple object to have
dimension greater than one. Note that one can satisfy
the equalities \eqref{eq-s6.2-3} and \eqref{eq-s6.2-2},
for example, if $\overline{\mathcal{X}}(N_2,N_2)$ has dimension two
while $F_2(N_1)\cong N_1\oplus N_3$
and $\overline{\mathcal{X}}(N_1,N_1)$ has dimension one.

\section{General results}\label{s10}

\subsection{Module categories inside
$\mathfrak{sl}_2$-mod generated by simple modules}\label{s10.1}

Let $\mathfrak{g}=\mathfrak{sl}_2$ with $\varphi$ being the identity. 
Consider the $\mathscr{C}$-module category 
$\mathrm{add}(\mathscr{C}\cdot M)$, where $M$ is a simple
$\mathfrak{sl}_2$-module. By Schur's lemma, 
see \cite[Proposition~2.6.8]{Di}, there is
$\vartheta\in\mathbb{C}$ such that $(\mathtt{c}-\vartheta)M=0$.
We can explicitly classify the types of simple transitive 
subquotients of $\mathrm{add}(\mathscr{C}\cdot M)$.

\begin{theorem}\label{thm-s10.1-1}
Under the above assumptions, we have:

\begin{enumerate}[$($a$)$]
\item\label{thm-s10.1-1.1} If $\vartheta$ is not the square of a
half-integer, then $\mathrm{add}(\mathscr{C}\cdot M)$ is a simple
transitive $\mathscr{C}$-module category of type $A_\infty^\infty$.
\item\label{thm-s10.1-1.4} If $\vartheta$ is the square of a
half-integer but not an integer and there is a special 
projective functor $\theta$ such that $\theta M\cong M$, 
then $\mathrm{add}(\mathscr{C}\cdot M)$  is a simple transitive 
$\mathscr{C}$-module category of type $T_\infty$.
\item\label{thm-s10.1-1.45} If $\vartheta$ is the square of a
half-integer but not an integer and the condition 
in \eqref{thm-s10.1-1.4} is not satisfied, then 
$\mathrm{add}(\mathscr{C}\cdot M)$  is a simple transitive 
$\mathscr{C}$-module category of type $A_\infty^\infty$.
\item\label{thm-s10.1-1.2} If $M$ is finite dimensional,
then $\mathrm{add}(\mathscr{C}\cdot M)$ is a simple
transitive $\mathscr{C}$-module category of type $A_\infty$.
\item\label{thm-s10.1-1.3} In all other cases, there is 
a short exact sequence 
\begin{displaymath}
0\to\mathcal{M}\to 
\mathrm{add}(\mathscr{C}\cdot M)\to
\mathcal{N}\to 0
\end{displaymath}
in the sense of \cite[Subsection~5.2]{CM}, where $\mathcal{M}$ is a simple
transitive $\mathscr{C}$-module category of type $C_\infty$
and  $\mathcal{N}$ is a simple
transitive $\mathscr{C}$-module category of type $A_\infty$.
\end{enumerate}
\end{theorem}

Note that, the situations described in 
both Claim~\eqref{thm-s10.1-1.4} and
Claim~\eqref{thm-s10.1-1.45} are possible,
see Subsections~\ref{s5.1} and \ref{s9.1}.

\begin{proof}
Let $M$ be a simple $\mathfrak{sl}_2$-module.
If $M$ is finite dimensional, then 
$\mathrm{add}(\mathscr{C}\cdot M)=\mathscr{C}$
and hence Claim~\eqref{thm-s10.1-1.2} follows 
from Proposition~\ref{prop-s4.2-1}.

Assume now that $M$ is infinite-dimensional and
let $L(\lambda)$ be a simple highest weight module
with the same annihilator in $U(\mathfrak{sl}_2)$ as $M$,
see \cite{Du}. This means, in particular, that 
$\vartheta=(\lambda+1)^2$. Assume first that 
$\lambda\not\in \frac{1}{2}\mathbb{Z}$.
Then 
\begin{equation}\label{eq-s10.1-5}
L(1)\otimes_{\mathbb{C}}M\cong\theta_{\lambda,\lambda+1}M
\oplus \theta_{\lambda,\lambda-1}M
\end{equation}
with both  $\theta_{\lambda,\lambda+1}M$ and
$\theta_{\lambda,\lambda-1}M$ simple (as both
$\theta_{\lambda,\lambda\pm 1}$ are equivalences of
appropriate categories) and non-isomorphic (as the two
modules have different annihilators). Moreover, the annihilators
of both $\theta_{\lambda,\lambda\pm 1}M$ are different from
the annihilator of $M$. Note that
both $\lambda\pm 1\not\in \frac{1}{2}\mathbb{Z}$. Therefore
we can continue this recursively, ending up with 
Claim~\eqref{thm-s10.1-1.1}.

Next let us assume that $\lambda\in\frac{1}{2}+\mathbb{Z}$.
Then we still have the decomposition \eqref{eq-s10.1-5}
with both summands being simple and non-isomorphic 
modules. However, if
$\lambda\in\{-\frac{1}{2},-\frac{3}{2}\}$, then one of the 
two modules $\theta_{\lambda,\lambda\pm 1}M$ will have the
same annihilator as $M$ and thus we might have the situation
that either $\theta_{\lambda,\lambda+ 1}M$ or
$\theta_{\lambda,\lambda-1}M$ is isomorphic to $M$.
This splits our situation in two cases (with both occurring,
see Subsections~\ref{s5.1} and \ref{s9.1}).

In the first case, when neither $\theta_{\lambda,\lambda+ 1}M$ nor
$\theta_{\lambda,\lambda-1}M$ is isomorphic to $M$,
we get that $\mathrm{add}(\mathscr{C}\cdot M)$ is of type
$A_\infty^\infty$ similarly to Subsection~\ref{s5.1}.
In the other case, that is when either $\theta_{\lambda,\lambda+ 1}M$ 
or $\theta_{\lambda,\lambda-1}M$ is isomorphic to $M$,
we get that $\mathrm{add}(\mathscr{C}\cdot M)$ is of type
$T_\infty$ similarly to Subsection~\ref{s9.1}. This proves
Claims~\eqref{thm-s10.1-1.4} and \eqref{thm-s10.1-1.45}.

It remain to consider the case $\lambda\in\mathbb{Z}_{<0}$.
In this case, due to our conventions on the notation for
projective functors, the decomposition \eqref{eq-s10.1-5} becomes
\begin{displaymath}
L(1)\otimes_{\mathbb{C}}M\cong\theta_{-\lambda-2,-\lambda-1}M
\oplus \theta_{-\lambda-2,-\lambda-3}M
\end{displaymath}
with both summands being simple and non-isomorphic 
modules provided that $\lambda\neq -1$. 
If $\lambda=-1$, then we have
\begin{displaymath}
L(1)\otimes_{\mathbb{C}}M\cong\theta_{-1,-2}M. 
\end{displaymath}
In this way, recursively, 
starting from $M$, we can construct simple modules $N(k)$,
for each $k\in\mathbb{Z}_{<0}$, with $M$ being one of these
modules, such that $N(k)$ has the same annihilator as $L(k)$
and 
\begin{displaymath}
L(1)\otimes_{\mathbb{C}}N(k)\cong N(k+1)\oplus N(k-1),
\end{displaymath}
for all $k\neq -1$. Using \cite[Theorem~67]{MS1}, we further have
that the module $N'(-2):=L(1)\otimes_{\mathbb{C}}N(-1)$ is indecomposable,
has the dual numbers as the endomorphism algebra, has
simple top and socle isomorphic to $N(-2)$ and no other 
infinite-dimensional simple subquotients.
Setting $N'(-1):=N(-1)\oplus N(-1)$, we can now recursively 
define the indecomposable modules $N'(k)$, for $k\leq -3$, via 
\begin{displaymath}
L(1)\otimes_{\mathbb{C}}N'(k+1)\cong N'(k)\oplus N'(k+2).
\end{displaymath}
Let $\mathcal{M}$ be the additive closure of all
$N'(k)$, where  $k\leq -1$. Let $\mathcal{N}$ be 
the additive closure of  all $N(k)$, for $k\leq -2$.
Then $\mathcal{M}$ is closed under the action of
$\mathscr{C}$ and the quotient of $\mathrm{add}(\mathscr{C}\cdot M)$
by the ideal generated by $\mathcal{M}$ is equivalent to 
$\mathcal{N}$. Therefore, we get
Claim~\eqref{thm-s10.1-1.3} by the same arguments as
in Subsection~\ref{s7.2}. This completes the proof.
\end{proof}

\subsection{Module categories coming from 
restrictions to subalgebras}\label{s10.2}

Let $\mathfrak{a}$ be a Lie subalgebra of $\mathfrak{sl}_2$.
Then $\mathscr{C}$ acts, via restriction, on 
the category of all finite dimensional $\mathfrak{a}$-modules.
In this subsection we will describe simple transitive 
$\mathscr{C}$-module categories of this kind generated by
simple finite dimensional $\mathfrak{a}$-modules.
In the extreme case $\mathfrak{a}=\mathfrak{sl}_2$, we
just get the left regular $\mathscr{C}$-module category
${}_\mathscr{C}\mathscr{C}$. The other extreme case
$\mathfrak{a}=0$ is trivial.

It remains to consider the two cases when 
$\mathfrak{a}$ has dimension $1$ or $2$. Let us start with
the case $\dim(\mathfrak{a})=1$. Then $\mathfrak{a}$
is the linear span of some non-zero element 
$g\in \mathfrak{sl}_2$, in particular, 
$U(\mathfrak{a})\cong\mathbb{C}[g]$, the polynomial algebra
in $g$. For $\lambda\in\mathbb{C}$, let $\mathbb{C}_\lambda$
denote the one-dimensional $\mathbb{C}[g]$-module on which
$g$ acts via $\lambda$.

We have the following two cases:
\begin{itemize}
\item the element $g$ is nilpotent,
\item the element $g$ is semi-simple with eigenvalues
$\mu\neq 0$ and $-\mu$.
\end{itemize}
In the first case, we have a non-split
short exact sequence 
\begin{equation}\label{eq-s10.2-4}
0\to \mathbb{C}_{\lambda}\to
(\mathrm{Res}^{\mathfrak{sl}_2}_\mathfrak{a}L(1))
\otimes_\mathbb{C}\mathbb{C}_\lambda\to
\mathbb{C}_{\lambda}\to 0.
\end{equation}
For $n\in\mathbb{Z}_{>0}$, denote by $M(n,\lambda)$
the $n$-dimensional $\mathbb{C}[g]$-module on which
$g$ acts via an $n\times n$ Jordan cell with 
eigenvalue $\lambda$. Let $\mathcal{Y}(\lambda)$ be the
additive closure of $M(n,\lambda)$, for all $n$,
and let $\widetilde{\mathcal{Y}}(\lambda)$ be the 
semi-simplification of $\mathcal{Y}(\lambda)$, that is the 
quotient of $\mathcal{Y}(\lambda)$ by its radical.

In the second case, we have 
\begin{equation}\label{eq-s10.2-3}
(\mathrm{Res}^{\mathfrak{sl}_2}_\mathfrak{a}L(1))
\otimes_\mathbb{C}\mathbb{C}_\lambda\cong
\mathbb{C}_{\lambda+\mu}\oplus
\mathbb{C}_{\lambda-\mu}.
\end{equation}
Let  $\mathcal{X}=\mathcal{X}(\lambda)$ denote 
the additive closure of all $\mathbb{C}_{\lambda+n\mu}$,
where $n\in \mathbb{Z}$.

\begin{proposition}\label{prop-s10.2-1}
We have the following two cases:
\begin{enumerate}[$($a$)$]
\item\label{prop-s10.2-1.1} If $g$ is nilpotent,
then, for each  $\lambda\in\mathbb{C}$, the 
category  $\mathcal{Y}(\lambda)$ is a transitive
$\mathscr{C}$-module category of type $A_\infty$
with simple transitive quotient 
$\widetilde{\mathcal{Y}}(\lambda)$.
\item\label{prop-s10.2-1.2} If 
$g$ is semi-simple, then, for each 
$\lambda\in\mathbb{C}$, the category 
$\mathcal{X}(\lambda)$ is a simple transitive 
$\mathscr{C}$-module category of type $A_\infty^\infty$.
\end{enumerate}
\end{proposition}

\begin{proof}
If $g$ is nilpotent, then  Formula~\eqref{eq-s10.2-4},
which can be written as
\begin{displaymath}
(\mathrm{Res}^{\mathfrak{sl}_2}_\mathfrak{a}L(1))
\otimes_\mathbb{C}M(1,\lambda)\cong M(2,\lambda)
\end{displaymath}
generalizes, by a direct computation, to 
\begin{displaymath}
(\mathrm{Res}^{\mathfrak{sl}_2}_\mathfrak{a}L(1))
\otimes_\mathbb{C}M(n,\lambda)\cong
M(n+1,\lambda)\oplus M(n-1,\lambda), 
\end{displaymath}
for $n>1$. This implies that $\mathcal{Y}(\lambda)$ is a transitive
$\mathscr{C}$-module category of type $A_\infty$.
In order to prove that the simple transitive 
quotient of $\mathcal{Y}(\lambda)$ is 
$\widetilde{\mathcal{Y}}(\lambda)$, it remains to show
that the radical $\mathcal{R}$ of $\mathcal{Y}(\lambda)$ is 
$\mathscr{C}$-invariant. Since $\mathscr{C}$
is generated by $L(1)$, we just need to show that 
$\mathcal{R}$ is invariant under the action of $L(1)$.

It is easy to see that $\mathcal{R}$ is generated by 
the projections $M(n,\lambda)\tto M(n-1,\lambda)$
and the injections $M(n,\lambda)\tto M(n+1,\lambda)$.
Applying $L(1)$ to the former, we get a projection
of $M(n-1,\lambda)\oplus M(n+1,\lambda)$ onto
$M(n,\lambda)\oplus M(n-2,\lambda)$
(or just onto $M(2,\lambda)$, if $n=2$),
which is obviously in $\mathcal{R}$ since the domain
and the codomain of this projection do not have 
isomorphic indecomposable summands. Similarly,
applying $L(1)$ to the latter, we get an injection
of $M(n+1,\lambda)\oplus M(n-1,\lambda)$ 
(or just of $M(2,\lambda)$, if $n=1$),
into $M(n+2,\lambda)\oplus M(n,\lambda)$,
which is obviously in $\mathcal{R}$ for the same
reasons as above. This implies that
$\mathcal{R}$ is $\mathscr{C}$-invariant
and completes the proof of Claim~\eqref{prop-s10.2-1.1}.

If $g$ is semi-simple, then $\mathcal{X}(\lambda)$
is a semi-simple and transitive, hence simple transitive, 
$\mathscr{C}$-module category of type $A_\infty^\infty$
thanks to Formula~\eqref{eq-s10.2-3}. 
Claim~\eqref{prop-s10.2-1.2} follows and the proof is complete.
\end{proof}

We proceed with the case $\dim(\mathfrak{a})=2$. 
Up to an inner automorphism, the only two-dimensional
subalgebra $\mathfrak{a}$ of $\mathfrak{sl}_2$
is the standard Borel subalgebra $\mathfrak{b}$ 
generated by $h$ and $e$.
All simple finite dimensional modules over this algebra
have dimension one and have the following form: 
for each $\lambda\in\mathbb{C}$, we have 
the unique one-dimensional $\mathfrak{b}$-module
$Q(\lambda)$ which is annihilated by $e$ and on which
$h$ acts via $\lambda$. The $\mathscr{C}$-module category
$\mathrm{add}(\mathscr{C}\cdot Q(\lambda))$
appears in Proposition~\ref{lem-s5.87-3}
where it was denoted by $\mathcal{N}_6$.
In Proposition~\ref{lem-s5.87-3} it was shown that 
this is a semi-simple simple transitive 
$\mathscr{C}$-module category of type $A_\infty$.
Moreover, as  a $\mathscr{C}$-module category,
it does not depend on $\lambda$, up to equivalence.

\vspace{2mm}

\noindent
V.~M.: Department of Mathematics, Uppsala University, Box. 480,
SE-75106, Uppsala,\\ SWEDEN, email: {\tt mazor\symbol{64}math.uu.se}

\noindent
X.~Z.: School of Mathematics and Statistics, 
Ningbo University, Ningbo, 
Zhejiang, 315211, China, email: 
{\tt zhuxiaoyu1\symbol{64}nbu.edu.cn}

\end{document}